\documentclass[11pt]{article}
\usepackage[a4paper, total={11.5in, 9.5in}]{geometry}
\usepackage{cite}
\usepackage{float}
\usepackage{pdflscape}
\usepackage{mathtools}  
\allowdisplaybreaks
\usepackage{amssymb,amsfonts} 
\usepackage{authblk}
\usepackage{fullpage}
\usepackage{graphicx}
\usepackage{lscape}
\usepackage{bm}
\usepackage{comment}
\usepackage{enumerate}
\usepackage{hyperref}
\usepackage[utf8]{inputenc} % allow utf-8 input
\usepackage{booktabs} 
\usepackage[T1]{fontenc}    % use 8-bit T1 fonts
\usepackage{url}            % simple URL typesetting
\usepackage{multirow}
\usepackage{amsthm}
\usepackage{amsmath}
\usepackage{siunitx}
\usepackage{subfigure}
\usepackage[textsize=tiny]{todonotes}

\newcommand{\p}{\partial}

\newtheorem{thm}{Theorem}[section]

\newtheorem{lemma}[thm]{Lemma}

\newcommand{\jump}[1]{[\![#1]\!]}
\newcommand{\con}{\textbf{U}}

\newcommand{\F}{\textbf{F}}

\newcommand{\ent}{\mathcal{E}}

\newcommand{\evar}{\textbf{V}}

\newcommand{\es}{\textbf{es}}
\newcommand{\explicit}{\textbf{exp}}

\newcommand{\ote}{\textbf{O2}^\es_\explicit}
\newcommand{\othe}{\textbf{O3}^\es_\explicit}
\newcommand{\ofe}{\textbf{O4}^\es_\explicit}

\newcommand{\bbb}[1]{\left[\!\left[#1\right]\!\right]}
\newcommand{\ccc}[1]{\left\{\!\!\left\{#1\right\}\!\!\right\}}

\usepackage{mdframed}
\usepackage[font=small]{caption}

\usepackage{placeins}

\title{Entropy dissipative high order schemes for a hyperbolic model of two-layer thin film flow}
\author{Rahul Barthwal$^1$, Philipp Öffner$^2$, Chetan Singh$^3$, and Jan Zawallich$^{2,\star}$}
\date{$^1$ Germany, University of Stuttgart, Institute of Applied Analysis and Numerical Simulation, Chair of Applied Mathematics\\ $^2$ Germany, Clausthal University of Technology, Institute of Mathematics, Numerical Analysis Research Group\\ $^3$ India, Indian Institute of Technology Delhi, Department of Mathematics\\ $^\star$ Corresponding author}

\begin{document}
\maketitle
\begin{abstract}
In this article, we develop high-order, entropy-dissipative schemes for a hyperbolic system of conservation laws describing the first-order dynamics of two-layer thin-film flows of immiscible fluids with perfectly soluble solute particles. The key aspect is to construct entropy-conservative fluxes in the sense of Tadmor. Here, we consider and compare two different construction processes: an exact formulation and a commonly used simplified approximation in the flux evaluation. While the approximate approach is simpler at the level of construction and is frequently employed in practice (also in other contexts), it is less structured, leading to increased computational cost in the implementation  compared to the exact formulation, which has been much more complicated to derive. 
We employ these fluxes within both an entropy-dissipative finite-difference framework and an entropy-dissipative discontinuous Galerkin spectral element method. Through numerical experiments using two distinct spatial discretizations, we investigate the performance and robustness of the proposed methods and demonstrate that the choice of numerical flux has a noticeable impact on  the efficiency of the resulting solvers.
\end{abstract}

\section{Introduction}
Thin film flows have been extensively investigated within the framework of hyperbolic balance laws, particularly for gravity-driven films, inclined-plane flows, and surfactant-mediated dynamics; see, e.g., \cite{barthwal2022two, bertozzi1999undercompressive, barthwal2023construction, pandey2025construction, levy2006motion, cook2008shock}. These models arise in numerous industrial and environmental applications, including coating technologies, thin-film photovoltaics, and biomedical surfactant therapies. 

Mathematically, thin film equations are typically formulated as fourth-order degenerate hyperbolic–parabolic systems, with degeneracy occurring as the film height vanishes. In appropriate asymptotic regimes, however, the large-scale lubrication dynamics can be reduced to a first-order hyperbolic subsystem; see, e.g., \cite{barthwal2023construction,barthwal2025hyperbolic, cook2008shock}.
For hyperbolic systems, an important additional physical structure is provided by an entropy inequality. In particular, for scalar conservation laws and for one-dimensional systems, entropy conditions play a crucial role in selecting the physically relevant weak solution and establishing well-posedness \cite{bressan2000hyperbolic,dafermos2005hyperbolic}. This is particularly important for the hyperbolic thin-film subsystem considered here, where discontinuities and shock-like structures may develop even from smooth initial data. The entropy inequality therefore provides a natural criterion for identifying the physically admissible solution among the potentially many weak solutions of the hyperbolic system. Only after establishing this entropy structure, it is meaningful to ask how it should be preserved at the discrete level.
Therefore, the entropy structure of this hyperbolic subsystem is a key ingredient in the design of stable and physically consistent numerical discretizations of the full model.

The development of entropy-dissipative schemes has attracted considerable research interest in recent years, spanning a wide range of discretization frameworks, from finite-difference methods to finite-element-type approaches; see, e.g., \cite{abgrall2022,abgrall2022re,oeffner2023,gaburro2023,fjordholm2012arbitrarily,bhoriya2023high, FISHER2013518, GASSNER201639,leFloch2002} and references therein. Most of these developments build upon the seminal works of Tadmor \cite{tadmor1987,tadmor2003entropy}, who established the fundamental framework for entropy-conservative numerical approximations and the construction of entropy-conservative two-point fluxes. Entropy-dissipative fluxes can subsequently be obtained by adding suitable dissipative terms to these entropy-conservative fluxes. In the context of thin-film flows, provably positivity-preserving and entropy-stable numerical methods based on an explicit entropy structure have been successfully developed for scalar, single-layer equations; see, e.g., \cite{zhornitskaya1999positivity, kim_positivity} and the references therein. However, extending these approaches to coupled single- and multi-layer systems remains largely open. 

In this article, we take a step in this direction by constructing entropy-dissipative high-order numerical schemes for a system of conservation laws describing the leading-order dynamics of two-layer thin-film flows of immiscible fluids under the influence of anti-surfactants. A key ingredient in the construction of entropy-stable schemes is the development of suitable entropy-conservative two-point fluxes in the sense of Tadmor. In particular, we derive and investigate two different sets of entropy-conservative two-point fluxes and use them as building blocks for the construction of high-order entropy-stable schemes. The first is based on a commonly used approximate evaluation of the entropy-conservative flux, while the second employs an exact evaluation. The approximate construction is considerably simpler and follows an approach that is also frequently used for other hyperbolic systems, including the compressible Euler equations \cite{winters2020,carlo2023}. In contrast, the exact construction requires a more involved derivation, but preserves an additional symmetry structure in flux evaluation. This structure has important consequences for the resulting discretization and leads to improved properties when solving the linear system arising at each time step, which has to be solved for both constructions. 
 Thus, the choice of entropy-conservative flux directly affects not only the complexity of the flux construction but also the computational cost, robustness, and stability of the resulting solver.
 We incorporate both flux formulations into an entropy-dissipative finite-difference (FD) framework~\cite{fjordholm2013eno,singh2025entropy,singh2026entropy,singh2026oftt}, as well as into an entropy-dissipative discontinuous Galerkin spectral element method (DGSEM) \cite{GASSNER201639, oeffner2023}. By considering these two distinct spatial discretizations, we can assess the influence of the flux construction independently of the underlying discretization and compare the resulting schemes in terms of computational efficiency and robustness. 
 We incorporate both flux formulations into an entropy-dissipative finite-difference (FD) framework  \cite{singh2025entropy,fjordholm2013eno}, as well as into an entropy-dissipative discontinuous Galerkin spectral element method (DGSEM) \cite{GASSNER201639, oeffner2023}. By considering these two distinct spatial discretizations, we can assess the influence of the flux construction independently of the underlying discretization and compare the resulting schemes in terms of computational efficiency and robustness. 
 
 The rest of the article is organized as follows: In Section \ref{sec_model}, we introduce the thin-film model and the underlying structural properties. Afterward, in Section \ref{se_methods}, we introduce the entropy-dissipative FD method and DGSEM and describe how entropy-conservative two-point fluxes can be constructed for the thin-film model.  In Section \ref{se_numerics}, we assess the proposed fluxes through numerical experiments for one- and two-dimensional problems, comparing their accuracy, robustness, and computational efficiency independent of the underlying discretization. A summary and outlook are then given. In the appendix, we provide additional calculations that are omitted from the main text for better readability.

\section{Model}\label{sec_model}

The thin-film model considered in this work is described in detail in \cite{barthwal2025hyperbolic}. It is given by the following one-dimensional hyperbolic system\footnote{We restrict our analysis here to the one-dimensional setting for the sake of clear presentation. Numerical results for the two-dimensional case are presented in Section \ref{se_numerics}. }:

%The hyperbolic system takes the  form
\begin{equation}\label{eq: Main_system}
\begin{aligned}
    \dfrac{\partial f}{\partial t}+\dfrac{1}{2}\dfrac{\partial}{\partial x}\left(f^2b\right)&=0,\vspace{0.2 cm} \\
    \dfrac{\partial b}{\partial t}+\dfrac{1}{2}\dfrac{\partial}{\partial x}\left(fb^2\right)&=0,\vspace{0.2 cm} \\
     \dfrac{\partial g}{\partial t}+\dfrac{\partial}{\partial x}\left(\dfrac{g^2q}{2}+fgb\right)&=0,\vspace{0.2 cm} \\
  \dfrac{\partial q}{\partial t}+\dfrac{\partial}{\partial x}\left(\dfrac{gq^2}{2}+fbq\right)&=0.
  \end{aligned}
\end{equation}
In \eqref{eq: Main_system}, $f$ and $g$ denote the positive film thicknesses in the two layers and the quantities 
%\[
$b$ %=\dfrac{\partial c_1}{\partial x},\quad\text{and}\quad 
 and
$q$ %=\dfrac{\partial c_2}{\partial x}
%\]
denote the spatial derivatives of the concentrations of the solute in each layer. 

The first-order system \eqref{eq: Main_system} can be written in the compact conservative form 
\begin{equation}\label{eq:conservation-laws}
\con_t+(F(\con))_x=0,
\end{equation}
where $\con=(f,b,g,q)^{\top}$ denotes the vector of conservative variables and 
\begin{align}\label{eq:flux}
F(\con)=\left(\frac 12 f^2b,\frac 12fb^2,\frac 12 g^2q+fbg,\frac 12 gq^2+fbq\right)^{\top}
\end{align}
is the flux vector.

Moreover, we consider the state space
\begin{align}\label{state_space}
    \mathcal{{U}}=\{(f, b, g,q )\in \mathbb{R}^4 : f, g, q>0,~ b\neq 0, ~gq>fb, ~2fb+gq>0\}.
\end{align}
Note that we take a more generalized state space in this article than the one used in \cite{barthwal2025hyperbolic}. This is motivated by the fact that the articles \cite{barthwal2025hyperbolic} and \cite{barthwal2025generalized} only consider a single signed condition on $b$ and thus consider only unidirectional dynamics of the solution behavior. We relax this condition in this article and consider the state space \eqref{state_space}, which automatically reduces to the state space considered in \cite{barthwal2025hyperbolic} and \cite{barthwal2025generalized} for $b>0$ and $b<0$, respectively.

Based on the sign of $b$, the eigenstructure of the system \eqref{eq: Main_system} satisfies the following ordering relations. 
For $b>0$,
\begin{align}\label{eigenvalues}
\lambda_1(\mathbf{U})=\dfrac{fb}{2}<~\lambda_2(\mathbf{U})=\dfrac{3fb}{2}<~\lambda_3(\mathbf{U})=fb+\dfrac{gq}{2}<~\lambda_4(\mathbf{U})=fb+\dfrac{3gq}{2}.
\end{align}
and for $b<0$
\begin{align}\label{eigenvalues_1}
\lambda_1(\mathbf{U})=\dfrac{3fb}{2}<~\lambda_2(\mathbf{U})=\dfrac{fb}{2}<~\lambda_3(\mathbf{U})=fb+\dfrac{gq}{2}<~\lambda_4(\mathbf{U})=fb+\dfrac{3gq}{2}.
\end{align}
Moreover, the (right) eigenvectors $\mathbf{r_k}= \mathbf{r_k}(\mathbf{U})$ for $b>0$ are
\begin{equation}\label{righteigenvectors}
\begin{array}{rcl}\displaystyle 
\mathbf{r_1}(\mathbf{U})= \left(-\frac{f}{b},1,0,0\right)^{\top}, && \!\!\!\!\!\! \displaystyle \mathbf{r_2}(\mathbf{U})=\left(\frac{fb-3gq}{4qb},\frac{fb-3gq}{4qf},\frac{g}{q},1\right)^{\top},\\[3ex]\displaystyle 
\mathbf{r_3}(\mathbf{U})=\left(0,0,-\frac{g}{q},1\right)^{\top}, & & \!\!\!\!\!\! \displaystyle \mathbf{r_4}(\mathbf{U})=\left(0,0,\frac{g}{q},1\right)^{\top}.
\end{array}
\end{equation}
The eigenvectors are linearly independent in $\mathcal U$.  This implies the strong hyperbolicity of the reduced system \eqref{eq: Main_system} in $\mathcal U$.

\subsection{Entropy/entropy pairs of the hyperbolic system \eqref{eq: Main_system}}
In this section, we first obtain a strictly convex entropy for the system \eqref{eq: Main_system} independent of the sign of $b$, which generalizes (while simplifying) the convex entropy obtained in \cite{barthwal2025hyperbolic}. We first recall the general class of entropies obtained in \cite{barthwal2025hyperbolic}.
\begin{thm}[Entropy/entropy-flux pairs] \label{theo1}
Let $\rho, \mu$ and $\nu$ be arbitrary functions in  $C^1(0,\infty)$.
For the system \eqref{eq: Main_system}, there exists a class of entropy/entropy-flux pairs $( \eta, \mathcal{F} )$ which is 
given by
\begin{equation}\label{eq:entropyclass}
\begin{array}{rcl}
 \eta[f, b, g, q]&=& \rho(fb)+\sqrt{fb}\mu\left(\dfrac{b}{\sqrt{fb}}\right)+\sqrt{gq}\nu\left(\dfrac{q}{\sqrt{gq}}\right)+\dfrac{1}{fb+gq},\\[3ex]
 \mathcal{F}[f, b, g, q]&=&\psi(fb)+(fb)^{3/2} \mu\left(\dfrac{b}{\sqrt{fb}}\right)+(\sqrt{gq}) \nu\left(\dfrac{q}{\sqrt{gq}}\right)(fb+gq)\\
&&\hspace{3.5 cm}-\ln \left((fb+gq)^{3/2}\right)-\dfrac{fb}{2(fb+gq)}.
\end{array}
\end{equation}
The function $\psi$ is determined from $\psi'(w)=\dfrac{3}{2}  w\rho'(w)$.% with $u=fb$.
\end{thm}
First, we recall that the system \eqref{eq: Main_system} is equipped with a full set of Riemann invariants \cite{barthwal2025hyperbolic} and thus the convexity condition of an entropy $\eta\in \ent$ gets simpler. It is sufficient to check that the Hessian  $H_{\eta}$ of $\eta$ satisfies the conditions (see Chapter 7.4 of \cite{dafermos2005hyperbolic})
\begin{equation} \label{reducedcriterion}
\mathbf{r}_k(\mathbf{U})^\top H_{\eta}(\mathbf{U})\, \mathbf{r_k(\mathbf{U})}>0 \qquad \text{for all }~k\in \{1, \ldots, 4\} \text{ and all $\mathbf{U}\in \mathcal U$.}
\end{equation}
We derive the following entropy-entropy flux pair for the system \eqref{eq: Main_system} in the following lemma.
\begin{lemma}
For any $\mathbf{U}\in \mathcal{U}$, the following pair of functions form an entropy-entropy pair:
\begin{equation}\label{Eq:Ent-EntFl pair}
\eta= \frac{1}{|fb|}+ \frac{1}{fb+gq},\qquad \mathcal{F}= -\frac{1}{2}\left(\frac{fb}{fb+gq} +3\ln(|fb|^{\sigma(b)})+3\ln(fb+gq)\right),% to be found
\end{equation}
where $\sigma(b)$ is the sign of $b$ and $\eta$ is a strictly convex function for $\mathbf{U}\in \mathcal{U}$. 
In particular, for smooth solutions of the hyperbolic system \eqref{eq: Main_system}, we have 
\begin{align}
    \frac{\partial}{\partial t}\eta+\frac{\partial}{\partial x}\mathcal{F}=0.
\end{align}
\end{lemma}
\begin{proof}
By choosing $\rho(u)=1/|u|$, $\mu=0$, $\nu=0$, it is easy to see that the function $\eta$ belongs to the general class of entropies of the system \eqref{eq: Main_system} along with the associated flux $\mathcal{F} $ as defined in \eqref{eq:entropyclass}. Therefore, we directly proceed to prove that the entropy $\eta$ is strictly convex for all $\mathbf{U}\in \mathcal{U}$. We separate the proof of convexity for $\eta$ into two different cases depending upon the sign of $b$.\\\\
\textbf{Case 1. $\mathbf{b>0}$}\\
For this case, we directly compute the Hessian matrix of $\eta$, which gives
    \begin{align*}
   H_{\eta}&= \left(\begin{array}{cccc}2\left(\frac{b^2}{(bf+gq)^3}+\frac{1}{f^3 b}\right)&\frac{bf-gq}{(bf+gq)^3}+\frac{1}{b^2f^2}&\frac{2bq}{(bf+gq)^3}&\frac{2bg}{(bf+gq)^3}\\ \frac{bf-gq}{(bf+gq)^3}+\frac{1}{b^2f^2}&2\left(\frac{f^2}{(bf+gq)^3}+\frac{1}{fb^3}\right)&\frac{2fq}{(bf+gq)^3}&\frac{2fg}{(bf+gq)^3}\\\frac{2bq}{(bf+gq)^3}&\frac{2fq}{(bf+gq)^3}&\frac{2q^2}{(bf+gq)^3}&\frac{gq-bf}{(bf+gq)^3}\\\frac{2bg}{(bf+gq)^3}&\frac{2fg}{(bf+gq)^3}&\frac{gq-bf}{(bf+gq)^3}&\frac{2g^2}{(bf+gq)^3}\end{array}\right),\\
   \end{align*}
Thus we have
\begin{align*}
   \mathbf{r_1}^\top H_{\eta}\mathbf{r_1}&= \frac{2f}{b(bf+gq)^2}+\frac{2}{b^3f},\\
   \mathbf{r_2}^\top H_{\eta}\mathbf{r_2}&=\frac{3(3gq-bf)\left(3g^2q^2+2bf(gq-fb)\right)}{8b^3f^3q^2(bf+gq)},\\
   \mathbf{r_3}^\top H_{\eta}\mathbf{r_3}&=\frac{2g}{q(bf+gq)^2},\\
   \mathbf{r_4}^\top H_{\eta}\mathbf{r_4}&=\frac{2g(3gq-bf)}{q(bf+gq)^3},
\end{align*}
   which are all positive for $\mathbf{U}\in \mathcal{U}$.\\\\
 \textbf{Case 2. $\mathbf{b<0}$}\\
For this case, we directly compute the Hessian matrix of $\eta$, which gives
   \begin{align*}
   \left(\begin{array}{cccc}
   \frac{2b^2}{(bf+gq)^3}-\frac{2}{bf^3}&\frac{bf-gq}{(bf+gq)^3}-\frac{1}{b^2f^2}&\frac{2bq}{(bf+gq)^3}&\frac{2bg}{(bf+gq)^3}\\ \frac{bf-gq}{(bf+gq)^3}-\frac{1}{b^2f^2}&\frac{2f^2}{(bf+gq)^3}-\frac{2}{b^3f}&\frac{2fq}{(bf+gq)^3}&\frac{2fg}{(bf+gq)^3}\\\frac{2bq}{(bf+gq)^3}&\frac{2fq}{(bf+gq)^3}&\frac{2q^2}{(bf+gq)^3}&\frac{gq-bf}{(bf+gq)^3}\\\frac{2bg}{(bf+gq)^3}&\frac{2fg}{(bf+gq)^3}&\frac{gq-bf}{(bf+gq)^3}&\frac{2g^2}{(bf+gq)^3}\\\end{array}\right),\\
    \end{align*}
Thus we have
\begin{align*}
   \mathbf{r_1}^\top H_{\eta}\mathbf{r_1}&= -\frac{2gq(2fb+gq)}{bf(bf+gq)^2},\\
   \mathbf{r_2}^\top H_{\eta}\mathbf{r_2}&=-\frac{3g(3gq-bf)(2bf+3gq)}{8b^3f^3q(bf+gq)},\\
   \mathbf{r_3}^\top H_{\eta}\mathbf{r_3}&=\frac{2g}{q(bf+gq)^2},\\
   \mathbf{r_4}^\top H_{\eta}\mathbf{r_4}&=\frac{2g(3gq-bf)}{q(bf+gq)^3},
\end{align*}
   which are all positive for $\mathbf{U}\in \mathcal{U}$.\\    
\end{proof}

We also define the entropy variables $\evar=\frac{\p\eta}{\p\mathbf{U}}$ for $b>0$ and $b<0$, respectively, which are given by
\begin{equation}\label{eq:envar}
    \evar = \left(\mp\frac{1}{b f^2}-\frac{b}{(bf+gq)^2},~\mp\frac{1}{b^2f} - \frac{f}{(bf+gq)^2},~- \frac{q}{(bf+gq)^2},~- \frac{g}{(bf+gq)^2}\right)^\top.
\end{equation}

\section{Constructing high-order entropy dissipative schemes}\label{se_methods}
In the following, we construct entropy-conservative or entropy-dissipative higher-order schemes by deriving numerical fluxes for \eqref{eq: Main_system}. For the sake of cleaner presentation, we present the calculation in this section for $b>0$. The case $b<0$ can be done analogouesly and we will later demonstrate also a numerical example for this. Here, we derive the entropy conservative fluxes using two different approaches. First, we use a Taylor approximation; afterward, we use an analytical approach. Before, we fix some notation and explain the basis on which these fluxes are derived. We also fix the notation for the space discretization. In particular, we introduce the entropy-dissipative FD scheme and DGSEM.

\subsection{Notation and background}
For a possibly discontinuous function $a\colon\Omega\to\mathbb{R}$ we define the jump via $\bbb{a}:=a^+-a^-$ and the average via $\ccc{a}:=\frac{1}{2}(a^++a^-)$, where $a^+$ and $a^-$ are the values at the discontinuity with $a^+=\lim_{x\to x_0^+}a(x)$ and $a^-=\lim_{x\to x_0^-}a(x)$. Further, let
\[
\ccc{a}_\text{log}:=\frac{\bbb{a}}{\bbb{\ln a}},
\]
be the logarithmic mean, which can be calculated in a numerically stable way following \cite{roe2009}. We refer also to \cite{kawai2025} for some recent approximations. Later on we need the identity.
$$
\jump{ab} = \ccc{a}\jump{b}+\ccc{b}\jump{a}.
$$

We now state Tadmor's approach \cite{tadmor2003entropy} for finding an entropy-conservative flux: Given a system of type \eqref{eq:conservation-laws}, a numerical flux is entropy-conservative if
\begin{equation}
0=\bbb{\evar}\cdot F^\text{num}-\bbb{\Psi(\con(\evar))}, \label{eq:tadmorsCondition}
\end{equation}
where $\evar={\left(\frac{\partial \eta}{\partial u_i}\right)}_{i=1}^n$, $\eta$ is an entropy and $\Psi=\evar\cdot F(\con)-\mathcal{F}$ is an entropy potential. It is given for the entropy-entropy flux pair \eqref{Eq:Ent-EntFl pair} by 
%\begin{itemize}
%\item For $b>0$: 
\begin{equation}\label{eq_potential_model}
\Psi= -2 + \frac{fb}{2(bf+gq)} + \frac{3}{2}\ln\left(bf(bf+gq)\right),
\end{equation}
%\item For $b<0$:
%\begin{equation*}
%\Psi= \frac{fb}{2(bf+gq)} + \frac{3}{2}\left\{\ln\left(bf+gq\right)-\ln\left(bf\right)\right\}
%\end{equation*}
%\end{itemize}
for which the derivation can be found in Appendix \ref{sec_analytical}.

For simplicity and later calculations, we introduce the following abbreviations: 
$
A:=fb,\quad B:=fb+gq,
$
then we have, 
$
\eta = \frac{1}{A}+\frac{1}{B}$ and $  \mathcal{F}=-\frac{1}{2}\frac{A}{B}-\frac{3}{2}\ln(AB).
$
Additionally, we have the following relations: 
\begin{align*}
v_1:=\frac{\partial\eta}{\partial f}&=-\frac{b}{(fb)^2}-\frac{b}{(fb+gq)^2}=-u_2(A^{-2}+B^{-2}),\\
v_2:=\frac{\partial\eta}{\partial b}&=-\frac{f}{(fb)^2}-\frac{f}{(fb+gq)^2}=-u_1(A^{-2}+B^{-2}),\\
v_3:=\frac{\partial\eta}{\partial g}&=-\frac{q}{(fb+gq)^2}=-u_4B^{-2},\\
v_4:=\frac{\partial\eta}{\partial q}&=-\frac{g}{(fb+gq)^2}=-u_3B^{-2}.
\end{align*}

Next, we briefly introduce the notation and discretization frameworks used throughout the remainder of this work, in particular the DGSEM and the entropy-dissipative finite-difference scheme.

\subsubsection*{Entropy Dissipative Discontinuous Galerkin Spectral Element Method}
 A central aspect is the use of entropy-conservative two-point fluxes within a flux-differencing framework to construct an entropy-dissipative (conservative) DG method.

We consider a one-dimensional spatial domain $\Omega=[x_L,x_R]$. For more details, we refer to the literature \cite{GASSNER201639,oeffner2023}. Then, the construction of the entropy dissipative DGSEM can be summarized as follows. 
We multiply the system \eqref{eq:conservation-laws} by a test function $\boldsymbol{\varphi}$ and integrate over the domain $\Omega$. Then, we divide $\Omega$ into non-overlapping elements $E_k=(x_k,x_{k+1})$, $k=1, \ldots, K$. Instead of working in each element separately, all calculations are done in the reference element $E_0 = (-1,1)$ and we create a mapping between the physical coordinates $x$ and computational coordinate $\chi$ on the reference element $E_0$ via
    \begin{equation*}
    \chi^k(x) = 2\left(\frac{x-x_{k}}{x_{k+1}-x_{k}}\right) - 1\in E_0,\, k=1,\ldots,K.
    \end{equation*}
    This introduces a Jacobian term on each element of the form
$
    \chi^k_x = \frac{2}{\Delta x_k}$ with $ \Delta x_k = x_{k+1} - x_k. 
$   From the mapping, we apply the chain rule to rewrite spatial derivatives from physical coordinates into the computational coordinates. The numerical solution and physical flux with nodal polynomials of degree $N$ are written in the Lagrange basis, e.g., 
    \begin{equation*}
    \con(x,t) \approx \sum_{j=0}^N\con_j(t)\ell_j(\chi),
    \end{equation*}
    where the interpolation nodes are taken to be $N+1$ Legendre-Gauss-Lobatto (LGL) nodes.  Now, we apply two times integration-by-parts, where, after the first round, the discontinuity at the element boundary is resolved by using a numerical flux in the spirit of FV methods. The second integration-by-parts yields the strong-form DGSEM. We now select the test function to be the Lagrange basis polynomials $\varphi = \ell_i(\chi)$, where $i=0,\ldots,N$ and approximate the integrals with LGL quadrature rules, where we additionally collocate the interpolation and quadrature nodes.   
    This finally yields the DGSEM. However, the discretization is not, in general, entropy conservative (dissipative). The key ingredient is now to replace the collocated flux in the volume with an entropy conservative, numerical two-point flux that extends a low-order entropy conservative FV flux to high-order~\cite{FISHER2013518} and results in a split-form discretization for the volume flux. 
For a uniform grid, the resulting semi-discrete nodal DGSEM takes the form
\begin{equation}\label{eq:semiDG}
\frac{\Delta x}{2}\dot{\con}_i + 2\sum_{j=0}^N\mathcal{D}_{ij}F^\text{num}(\con_i,\con_j) + \frac{\tau_i}{\omega_i}\left(F^\text{num, surf} - F\right) = 0,\quad\text{where}\quad i = 0,\ldots, N,
\end{equation}
with derivative matrix $\mathcal{D}$ approximating the first derivative $\frac{\partial}{\partial x}$, $\omega_i$ be the Gauss-Lobatto quadrature weights (which defines as well a diagonal norm mass matrix $\mathcal{M}=  \text{diag}(\omega_0,\ldots,\omega_N) $), and $\tau_i$ represent the contribution at the boundary. We have $$\tau_i=\begin{cases}
-1 \text{ for } i=0,\\
1 \; \text { for } i=N, \\
0 \;\text{ else}.
\end{cases}
$$

The choice of the symmetric, two-point volume flux $F^{\text{num}}(\con_i,\con_j)$ and the two-point surface flux $ F^{\text{num,surf}}$ plays a crucial role in ensuring the desired properties of the overall discretization. This has been demonstrated in several works, see, e.g., \cite{GASSNER201639}, building on the original analysis of Fisher et al. \cite{FISHER2013518}. In particular, it was shown that, under the appropriate choice of volume and surface fluxes, the resulting discretization inherits the desired stability and conservation properties. A summary of these results, adapted from \cite{oeffner2023, chen2017}, is given below:

\begin{thm}
    If $F^{\text{num}}(\con_i,\con_j) $ is consistent and symmetric, then \eqref{eq:semiDG} is conservative and high-order accurate. If we further assume that $F^{\text{num}}(\con_i,\con_j)$ is entropy conservative, then the semi-discrete scheme \eqref{eq:semiDG} is also entropy conservative within a single element. In addition, if the numerical flux $F^\text{num, surf}$ at the element interface is entropy conservative (dissipative), then the scheme is entropy conservative (dissipative). 
\end{thm}
Therefore, to construct entropy dissipative DG schemes within the framework, the development of entropy conservative two-point fluxes is essential.

\subsubsection*{Entropy Dissipative Finite Difference Scheme}

Next, we briefly introduce the entropy-dissipative finite-difference discretization used in the following. We consider the one-dimensional spatial domain $\Omega=[x_L,x_R]$, which is discretized by a uniform grid with grid points
$
x_i=x_L+i\Delta x,\; i=0,\ldots,N.
$
The semi-discrete FD approximation is given by
\begin{equation}\label{eq_semi}
\frac{\mathrm{d}}{\mathrm{d}t}\con_i(t)
+
\frac{
\F^{\mathrm{num}}_{i+\frac12}(t)-\F^{\mathrm{num}}_{i-\frac12}(t)
}{\Delta x}
=0,
\end{equation}
where $\F^{\mathrm{num}}_{i+\frac12}$ denotes a numerical flux consistent with the physical flux $F(\con)$.
Following \cite{tadmor2003entropy}, we have the following result.

\begin{thm}\label{thm:entropy-conservative-fd}
Assume that the numerical flux
$
\F^{\mathrm{num}}_{i+\frac12}
=
\F^{\mathrm{num}}(\con_i,\con_{i+1})
$
is consistent with the physical flux $F(\con)$ and satisfies Tadmor's entropy-conservation condition
$
\bbb{\evar}_{i+\frac12}
\cdot
\F^{\mathrm{num}}_{i+\frac12}
=
\bbb{\Psi}_{i+\frac12}.
$
Then the semi-discrete FD scheme
\eqref{eq_semi} is entropy conservative, i.e.,
\begin{equation}\label{eq:entropy-equality-fd}
\frac{\mathrm d}{\mathrm dt}\eta(\con_i)
+
\frac{
\widehat{\mathcal{F}}_{i+\frac12}
-
\widehat{\mathcal{F}}_{i-\frac12}
}{\Delta x}
=0,
\end{equation}
where the numerical entropy flux is defined by
\begin{equation}\label{eq:numerical-entropy-flux-fd}
\widehat{\mathcal{F}}_{i+\frac12}
=
\ccc{\evar}_{i+\frac12}
\cdot
\F^{\mathrm{num}}_{i+\frac12}
-
\ccc{\Psi}_{i+\frac12}.
\end{equation}
\end{thm}

However, the entropy-conservative fluxes derived from the Tadmor condition are only second-order accurate. Following~\cite{leFloch2002}, for any positive integer $p$, we can construct accurate fluxes of order $2p$ using second-order fluxes. For $p=2$, the fourth-order entropy-conservative flux $\tilde{\F}_{i+\frac{1}{2}}^4$ is given as follows:
    \begin{align}
\tilde{\F}_{i+\frac{1}{2}}^{\mathrm{num},4}&=\frac{4}{3}\,\F^{\mathrm{num}}(\con_{i},\con_{i+1})-\frac{1}{6} \bigg(\F^{\mathrm{num}} (\con_{i-1},\con_{i+1})+
	\F^{\mathrm{num}}(\con_{i},\con_{i+2}) \bigg)\label{eq:4thorder_numflux_x}
\end{align}
%Similarly, the fourth-order entropy-conservative flux $\tilde{\F^4}_{y,i,j+\frac{1}{2}}$ in $y$-direction as follows: \PO{We do not two-d currently. or what is the reason for this. }
%\begin{align}
%		\tilde{\F}_{y,i,j+\frac{1}{2}}^4&=\frac{4}{3}\tilde{\F}_{y}(\con_{i,j},\con_{i,j+1})-\frac{1}{6} \bigg( \tilde{\F}_{y} (\con_{i,j-1},\con_{i,j+1})+
%		\tilde{\F}_{y}(\con_{i,j},\con_{i,j+2}) \bigg).\label{eq:4thorder_numflux_y}
%\end{align}
%By replacing the second-order fluxes with the fourth-order fluxes and using fourth-order central difference approximations to approximate $\left(\frac{\p \con}{\p x}\right)_{i,j}$ and $\left(\frac{\p \con}{\p y}\right)_{i,j}$ given in~\eqref{eq:semi-discrete_fd}, we obtain fourth-order entropy-conservative schemes. According to~\cite{leFloch2002}, only even-order accurate fluxes can be constructed. Therefore, to construct an $m^{th}$-order accurate scheme, we use a $(m+1)^{th}$-order accurate flux, where $m$ is an odd integer. \PO{this can be changed! Indeed but it is only since it uses a central stenicl, we can also construct not ecen order by weighted stencisl... but this is a different topic}
Therefore, also in such context, higher-order methods can be constructed. 

Entropy-conservative schemes, while satisfying a discrete entropy equality, generally
lack the dissipation required to control oscillations in the presence of shocks. To obtain
an entropy-dissipative scheme, we therefore augment the
entropy-conservative flux by a suitable dissipative term. Following
\cite{tadmor2003entropy} again, we consider the modified numerical flux
\begin{equation}\label{eq:entropy-stable-flux}
\F^{\mathrm{es}}_{i+\frac12}
=
\F^{\mathrm{num}}_{i+\frac12}
-
\frac{1}{2}
\mathcal{D}_{i+\frac12}
\bbb{\evar}_{i+\frac12},
\end{equation}
where $\mathcal{D}_{i+\frac12}$ is a symmetric positive semidefinite matrix.
We obtain therefore the estimation
\begin{align}
\bbb{\evar}_{i+\frac12}
\cdot
\F^{\mathrm{es}}_{i+\frac12}
-
\bbb{\Psi}_{i+\frac12}
&=
-\frac{1}{2}
\bbb{\evar}_{i+\frac12}^{T}
\mathcal{D}_{i+\frac12}
\bbb{\evar}_{i+\frac12}
\leq 0.
\end{align}
Consequently, the resulting semi-discrete scheme is entropy dissipative. 

The entropy-stable flux contains the jump of the entropy variables and is therefore, in its basic form, only first-order accurate. To obtain higher-order entropy-dissipative schemes, we follow the sign-preserving reconstruction approach of \cite{fjordholm2012arbitrarily, fjordholm2013eno} as it is also used in \cite{singh2025entropy}. The key idea is to reconstruct suitably scaled entropy variables rather than the entropy variables directly. More precisely, we introduce entropy-scaled right eigenvectors of the symmetrized quasilinear form and use them to define the scaled entropy variables. The construction and scaling of these eigenvectors are given in Appendix~\ref{Eiegn_vector_for_con}. The same entropy-scaled eigenvectors are used to construct the diffusion matrix, which takes the form
\begin{equation}
\mathcal{D}_{i+\frac12}
=
\widetilde{R}_{i+\frac12}
\left|\Lambda_{i+\frac12}\right|
\widetilde{R}_{i+\frac12}^{T},
\end{equation}
where $\left|\Lambda_{i+\frac12}\right|$ contains the absolute values of the characteristic speeds. We then apply a sign-preserving higher-order reconstruction to the scaled entropy variables and transform the reconstructed values back to entropy variables. Denoting the resulting reconstructed jump by $\jump{\widehat{\evar}}^{\,k}_{i+\frac12}$, the higher-order entropy-stable flux is defined as
\begin{equation}
\F^{\mathrm{es},k}_{i+\frac12}
=
\F^{\mathrm{ec},2p}_{i+\frac12}
-
\frac12
\mathcal{D}_{i+\frac12}
\jump{\widehat{\evar}}^{\,k}_{i+\frac12}.
\end{equation}
Here, $\F^{\mathrm{ec},2p}$ denotes the corresponding higher-order entropy-conservative flux and the sign-preserving property of the reconstruction ensures that the dissipative contribution retains the correct sign and thus preserves the entropy-stability property. 
Similar to DGSEM, the key ingredient to develop an entropy dissipative high-order scheme is the existence of entropy conservative two point fluxes, which we will develop in the following using two slightly different approaches. 

\subsection{Entropy conservative flux using Taylor expansion}
\label{Sec_Tadmor}
%==============================================================================
%==============================================================================
The aim is to construct an entropy-conservative numerical flux satisfying Tadmor's condition \eqref{eq:tadmorsCondition}. For notational simplicity, we suppress the mesh indices and denote the left and right states by $\con_l$ and $\con_r$, respectively. When $a_l \approx a_r$, we employ the Taylor expansion-based approximation of the logarithmic mean proposed in \cite{roe2009}. Following the approach from  \cite{chandrashekar2013kinetic,ranocha2018}, we express the jumps in the entropy variables $\evar$ in terms of the jumps of $b$, $f$, $g$, and $q$ as

\begin{align*}
	[\![ \evar]\!]=\begin{pmatrix*}
		\left[\!\!\left[v_1\right]\!\!\right]\\ 
		\left[\!\!\left[v_2\right]\!\!\right]\\
		\left[\!\!\left[v_3\right]\!\!\right]\\
		\left[\!\!\left[v_4\right]\!\!\right]\\
	\end{pmatrix*} =\begin{pmatrix*}
		\left[\!\!\left[\mp\frac{1}{b f^2}-\frac{b}{(bf+gq)^2}\right]\!\!\right]\\[4pt]
		\left[\!\!\left[\mp\frac{1}{b^2f} - \frac{f}{(bf+gq)^2}\right]\!\!\right]\\[4pt]
		\left[\!\!\left[- \frac{q}{(bf+gq)^2}\right]\!\!\right]\\[4pt]
		\left[\!\!\left[- \frac{g}{(bf+gq)^2}\right]\!\!\right]\\
	\end{pmatrix*}
\end{align*}
Next, we expend the jumps of the variables and obtain the following expressions: 
\begin{align*}
	\left[\!\!\left[v_1\right]\!\!\right] &= \pm\frac{\ccc{\left(\frac{1}{b f^2}\right)}}{\ccc{b f^2}} \left\{2 \ccc{b}\ccc{f}[\![f]\!] + \ccc{f^2}[\![b]\!]\right\} \\ &-  \frac{[\![b]\!] - 2~ \ccc{(bf+gq)} ~\ccc{\left(\frac{b}{(bf+gq)^2}\right)}\left\{\ccc{b}[\![f]\!] + \ccc{f}[\![b]\!] + \ccc{g}[\![q]\!] + \ccc{q}[\![g]\!]\right\}}{\ccc{(bf+gq)^2}}\\
	\left[\!\!\left[v_2\right]\!\!\right] &= \pm\frac{\ccc{\left(\frac{1}{b^2 f}\right)}}{\ccc{b^2 f}} \left\{2 \ccc{b}\ccc{f}[\![b]\!] + \ccc{b^2}[\![f]\!]\right\} \\ &- \frac{[\![f]\!] - 2~ \ccc{(bf+gq)} ~\ccc{\left(\frac{f}{(bf+gq)^2}\right)}\left\{\ccc{b}[\![f]\!] + \ccc{f}[\![b]\!] + \ccc{g}[\![q]\!] + \ccc{q}[\![g]\!]\right\}}{\ccc{(bf+gq)^2}}\\
	\left[\!\!\left[v_3\right]\!\!\right] &= - \frac{[\![q]\!] - 2~ \ccc{(bf+gq)} ~\ccc{\left(\frac{q}{(bf+gq)^2}\right)}\left\{\ccc{b}[\![f]\!] + \ccc{f}[\![b]\!] + \ccc{g}[\![q]\!] + \ccc{q}[\![g]\!]\right\}}{\ccc{(bf+gq)^2}}\\
	\left[\!\!\left[v_4\right]\!\!\right] &= - \frac{[\![g]\!] - 2~ \ccc{(bf+gq)} ~\ccc{\left(\frac{g}{(bf+gq)^2}\right)}\left\{\ccc{b}[\![f]\!] + \ccc{f}[\![b]\!] + \ccc{g}[\![q]\!] + \ccc{q}[\![g]\!]\right\}}{\ccc{(bf+gq)^2}}.\\
\end{align*}
Let us consider the numerical flux $F^\text{num}= \left[ \tilde{F}^{(1)},\tilde{F}^{(2)}, \tilde{F}^{(3)}, \tilde{F}^{(4)}\right]^{\top}$
Then, we have 
\begin{equation}\label{eq_jump}
	\jump{\evar}\cdot F^\text{num} = \sum_{i=1}^{4}\jump{v_i}\cdot \tilde{F}^{(i)} = \jump{\Psi}.
\end{equation}
We can sort the jump expansions of \eqref{eq_jump} in terms of jumps in the conservative variables, e.g. $\jump{\con}=
(\bbb{f},\bbb{b},\bbb{g},\bbb{q})^T$ as well as 
the jump in the potential. By elemenary calculation, we obtain a  linear system of equations to calculate the entropy conservative flux components inside $F^{\textrm{num}} $. 
The resulting system can be written in matrix-vector from 
$\mathbf{A}F^{\textrm{num}}=\mathbf{b}$, and the solution of this linear system is given by $F^{\textrm{num}}=\mathbf{A^{-1}b}$. The details and the entries can be found in Appendix \ref{se_Taylor}. Please note that the resulting $4\times 4$ system has to be solved in each time-step to calculate the entropy conservative numerical two-point flux which is also done numerically inside our algorithm.   

%==============================================================================
%==============================================================================
\subsection{Analytical derivation of numerical flux}\label{sec_Analytical}
In the next part, we will derive an analytical derivation of the numerical flux without using Taylor series expansion for simplicity. We work directly with the flux formulation.

We rewrite \eqref{eq:tadmorsCondition} in terms of the conservative variables $\con$. Using the chain rule to express the averages and jumps in terms of the conservative variables, we factor out the jumps of all conservative variables. Exploiting the linear independence of the jumps of the components of $\con$, i.e.,
\[
\left(0=\sum_{i=1}^4\alpha_i\bbb{u_i}\right) \quad \Longrightarrow \quad \forall i\in\{1,2,3,4\}:\, \alpha_i=0.
\]
We thus obtain the following linear system of equations, whose solution determines the entropy-conservative numerical flux: 
\begin{equation}\label{eq_system}
\mathcal{A}F^\text{num}=c,
\end{equation}
with
\[
\mathcal{A}=\begin{pmatrix}\ccc{u_2}\ccc{\tilde{C}}&\ccc{u_1}\ccc{\tilde{C}}-\frac{\ccc{\tilde{D}}}{2\ccc{u_2}}&\ccc{u_4}\ccc{\tilde{B}}&\ccc{u_3}\ccc{\tilde{B}}\\
\ccc{u_2}\ccc{\tilde{C}}-\frac{\ccc{\tilde{D}}}{2\ccc{u_1}}&\ccc{u_1}\ccc{\tilde{C}}&\ccc{u_4}\ccc{\tilde{B}}&\ccc{u_3}\ccc{\tilde{B}}\\
\ccc{u_2}\ccc{\tilde{B}}&\ccc{u_1}\ccc{\tilde{B}}&\ccc{u_4}\ccc{\tilde{B}}&\ccc{u_3}\ccc{\tilde{B}}-\frac{\ccc{B^{-2}}}{2\ccc{u_4}}\\
\ccc{u_2}\ccc{\tilde{B}}&\ccc{u_1}\ccc{\tilde{B}}&\ccc{u_4}\ccc{\tilde{B}}-\frac{\ccc{B^{-2}}}{2\ccc{u_3}}&\ccc{u_3}\ccc{\tilde{B}}\end{pmatrix}
\]
and
\[
c=\frac{1}{4}\begin{pmatrix}
\ccc{B^{-1}}-\ccc{A}\frac{\ccc{B^{-1}}}{\ccc{B}}-3(\ccc{A}+\ccc{B})\ccc{AB}_{\log}^{-1}\\
\ccc{B^{-1}}-\ccc{A}\frac{\ccc{B^{-1}}}{\ccc{B}}-3(\ccc{A}+\ccc{B})\ccc{AB}_{\text{log}}^{-1}\\
-\ccc{A}\frac{\ccc{B^{-1}}}{\ccc{B}}-3\ccc{A}\ccc{AB}_{\log}^{-1}\\
-\ccc{A}\frac{\ccc{B^{-1}}}{\ccc{B}}-3\ccc{A}\ccc{AB}_{\log}^{-1}
\end{pmatrix},
\]
where $\ccc{\tilde{C}}:=\ccc{\tilde{A}}+\ccc{\tilde{B}}$, $\ccc{\tilde{D}}:=\ccc{A^{-2}+B^{-2}}$, $\ccc{\tilde{X}}=\ccc{X}\frac{\ccc{X^{-2}}}{\ccc{X^2}}$, and $X\in\{A,B\}$.  %and $\ccc{u}_{\log}:=\frac{u^+-u^-}{\log(u)^+-\log(u)^-}$, which can be evaluated in a numerical stable way following the approach from \cite{TODO}.
The details of the derivation are provided in Appendix~\ref{sec_analytical}.
The system $\mathcal{A}F^{\text{num}}=c$ with $F^{\text{num}}=(F_1,F_2,F_3,F_4)^T$ can be simplified to
\begin{equation}\label{eq:f1}
\ccc{u_1}F_2=\ccc{u_2}F_1,\quad \ccc{u_3}F_4=\ccc{u_4}F_3
\end{equation}
and
\begin{equation}\label{eq:f2}
\tilde{\mathcal{A}}\begin{pmatrix}F_1\\F_3\end{pmatrix}:=\begin{pmatrix} 2\ccc{u_2}\ccc{\tilde{C}}-\frac{\ccc{\tilde{D}}}{2\ccc{u_1}} & 2\ccc{u_4}\ccc{\tilde{B}} \\ 2\ccc{u_2}\ccc{\tilde{B}} & 2\ccc{u_4}\ccc{\tilde{B}}-\frac{\ccc{B^{-2}}}{2\ccc{u_3}}\end{pmatrix}\begin{pmatrix} F_1\\F_3\end{pmatrix}=\begin{pmatrix} c_1 \\ c_3\end{pmatrix}.
\end{equation}
If $\det(\tilde{\mathcal{A}})\neq0$, the system \eqref{eq:f2} (and \eqref{eq:f1}) can be uniquely solved via
\begin{align*}
F_1&=\frac{1}{\det(\tilde{\mathcal{A}})}\big((2\ccc{u_4}\ccc{\tilde{B}}-\frac{\ccc{B^{-2}}}{2\ccc{u_3}})c_1-2\ccc{u_4}\ccc{\tilde{B}}c_3\big), & F_2&=\frac{\ccc{u_2}}{\ccc{u_1}}F_1.\\
F_3&=\frac{1}{\det(\tilde{\mathcal{A}})}\big((2\ccc{u_2}\ccc{\tilde{C}}-\frac{\ccc{\tilde{D}}}{2\ccc{u_1}})c_3-2\ccc{u_2}\ccc{\tilde{B}}c_1\big),
& F_4&=\frac{\ccc{u_4}}{\ccc{u_3}}F_3.
\end{align*}
This gives us an analytical expression of the solution of the linear system \eqref{eq:f2} and thus an entropy-conservative two-point flux. 
We have now two different linear systems to solve to derive/determine the entropy-conservative two-point fluxes. %Applying these fluxes 
\section{Numerics}\label{se_numerics}

In the following, we test and investigate the properties of the derived schemes. In particular, we focus on the performance of the two derived numerical fluxes when used within the entropy-dissipative finite-difference (FD) scheme and the DGSEM framework.\\
We consider one- and two-dimensional numerical experiments\footnote{A similar construction procedure can be followed to derive the corresponding two-dimensional fluxes. These fluxes are also provided in the accompanying repository.}.
For the implementation of the DGSEM, we use the \texttt{Trixi.jl} framework \cite{ranocha2021adaptive}. For non-smooth test cases, we apply the subcell shock-capturing procedure introduced in \cite{hennemann2021provably}. Within the flux-differencing formulation, we use our two derived entropy-conservative fluxes as volume fluxes. For the surface flux, we consider either a local Lax--Friedrichs flux or the corresponding entropy-conservative two-point flux.
For the entropy-dissipative FD scheme, we extend the approach used in \cite{singh2025entropy} by incorporating the entropy-conservative two-point fluxes derived in this work. In addition, we use the Roe-type dissipation term described in Section~\ref{se_methods}.
The time discretization is usually performed using a strong-stability-preserving Runge-Kutta (SSP RK) scheme. The particular choice of the time integration method will be specified for the individual numerical experiments.\\
For the numerical simulations, we focus here on four test cases. We first consider a smooth problem and two Riemann problems in one spatial dimension. Finally, we consider a two-dimensional problem from \cite{biswas_2026_Riemann}, which demonstrates the overall performance of the derived numerical fluxes and their application within the two considered discretization frameworks.

\subsection{Smooth traveling wave solutions - convergence test}\label{Pro:Smooth_1d}

To demonstrate the high-order accuracy of our scheme, we consider a smooth traveling wave  for the system \eqref{eq: Main_system}. The exact evaluation are described in the work \cite{barthwal2025generalized}. Here, we give only the basic setting and  focus on  the following $2\pi$-periodic solution of the system \eqref{eq: Main_system} 
\begin{align}\label{eq: travelling_wave}
    f(x, t)&= 2+\sin(x+t),\, b(x, t)=\dfrac{2}{f(x, t)}, \, g(x, t)= 2, \, q(x, t)= 1.
\end{align}
The domain is therefore given by $[0,2\pi]$ with periodic boundary conditions and the initial data are also described in \eqref{eq: travelling_wave} (with $t=0$). \\

\textbf{DGSEM:} We compare the numerical result for varying mesh size for the spatial domain to the solution at time $t=3$ also given by \eqref{eq: travelling_wave}. In order to measure only the effects of different mesh sizes through the spatial discretization on the error, we use a high-order time integrator to minimize the error in time, in this case, a $3$rd-order strong stability preserving Runge-Kutta scheme with adaptive time step size with 9 stages is applied \cite{rackauckas2017differentialequations}. We use for the  surface fluxes the entropy conservative two-point fluxes. 

\begin{figure}[ht]
\begin{center}
\includegraphics[width=0.49\textwidth]{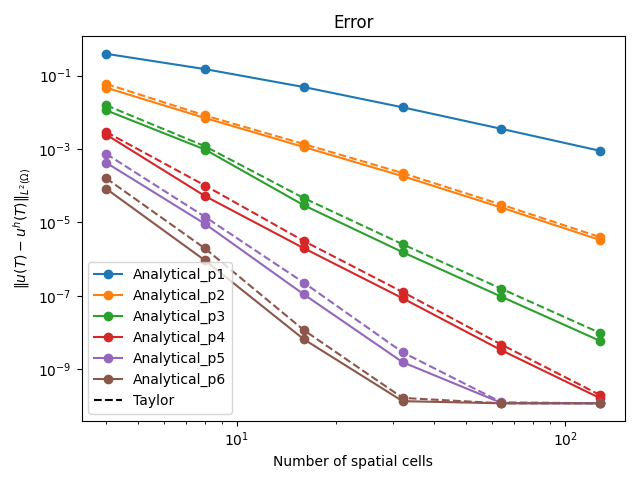}\hfill\includegraphics[width=0.49\textwidth]{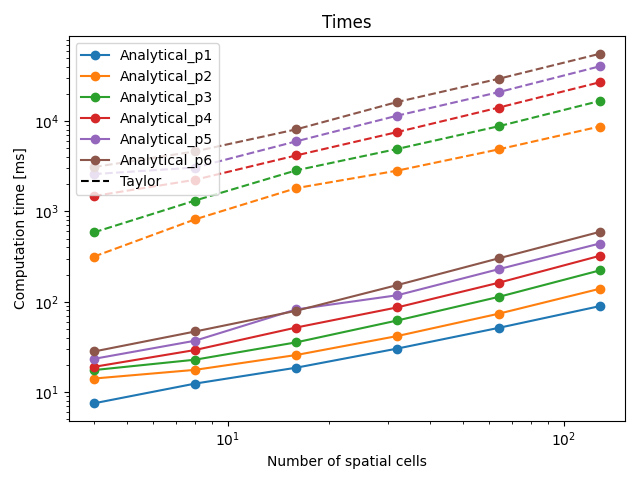}
\caption{The discrete $L^2(\Omega)$-error between $\con^h(T)$ and $\con (T)$ (left) and the computation time (right), both plotted in a log-log graph for different polynomial degress over an exponential decreasing spatial step size $h$.}\label{fig:errorTimes}
\end{center}
\end{figure}

\begin{figure}[!htbp]
	\begin{center}
		\subfigure[CPU time versus $L_1$ error using second order ED-FD scheme]{\includegraphics[width=0.3\textwidth]{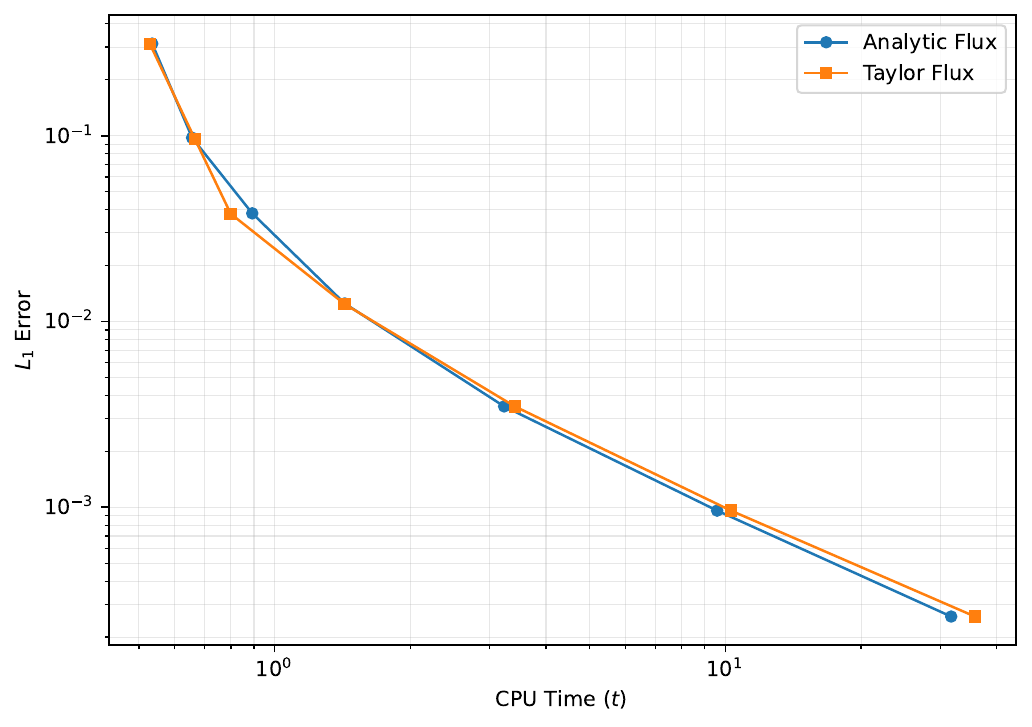} \label{fig:CPU_O2}}~
		\subfigure[CPU time versus $L_1$ error using third order ED-FD scheme]{\includegraphics[width=0.3\textwidth]{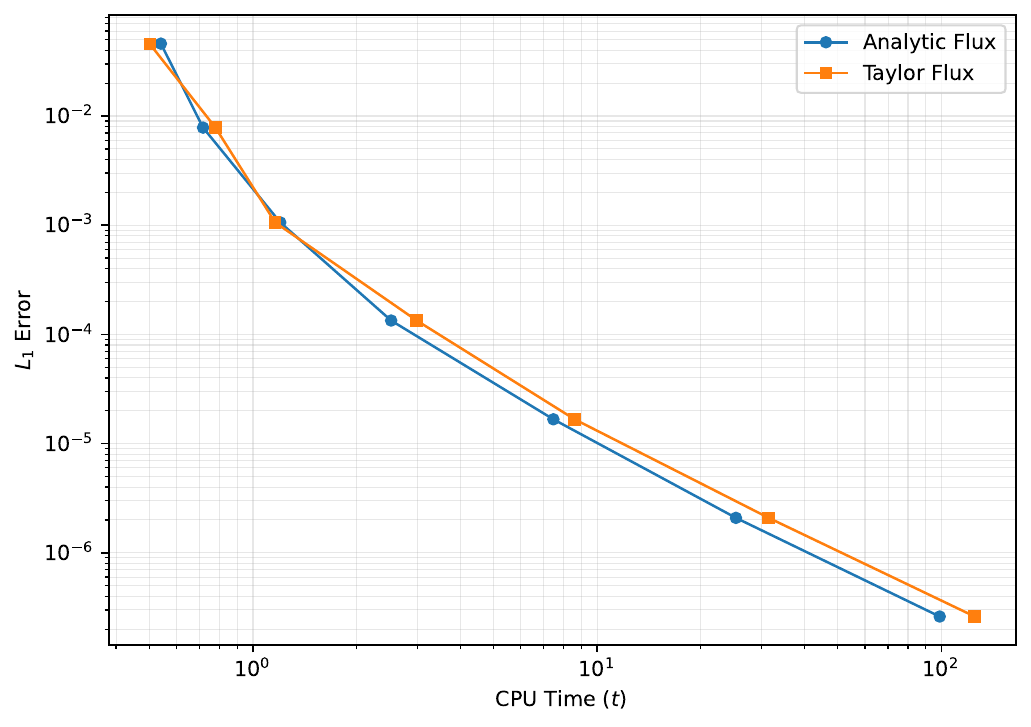} \label{fig:CPU_O3}}~
        \subfigure[CPU time versus $L_1$ error using fourth order ED-FD scheme]{\includegraphics[width=0.3\textwidth]{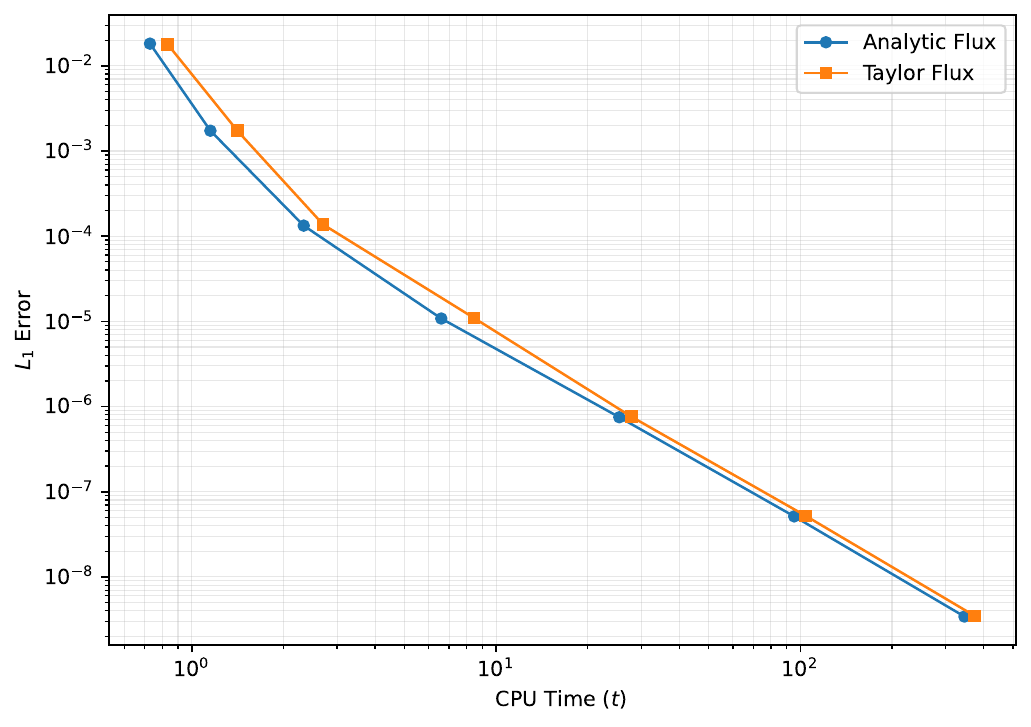} \label{fig:CPU_O4}}~
		\caption{\textbf{\nameref{Pro:Smooth_1d}}: Time vs accuracy comparison plots using the entropy-dissipative FD schemes with analytical and Taylor fluxes.}
		\label{fig:Cpu_VS_L1_error}
	\end{center}    
\end{figure}

The results for the discrete $L^2$-error and the computation time can be seen in Fig.~\ref{fig:errorTimes}. One can see that the error decreases with order $h^{p+1}$, where $h$ is the spatial step size and $p$ is the polynomial degree. One can also recognize that the errors calculated with the flux based on the Taylor expansion (denoted by Taylor in the Figure with dotted line) has a slightly larger error constant than the analytically derived flux. While the difference in error can be neglected for the two different fluxes, the computation time differs significantly: the analytical approach is at least one magnitude of order (and up to two orders of magnitude) faster than the Taylor approach. For every application of the flux, a $4\times 4$-matrix has to be inverted for the Taylor approach. For the analytical approach, only floating-point operations have to be performed that are much less computational heavy than inverting a matrix.
Additionally, we mark that the $p=1$ DGSEM scheme with the approximated flux is already breaking for this test case.

\textbf{Entropy dissipative FD: }
In Tables \ref{tab:O2_f}- \ref{tab:O4_f}, we report the errors for $f$ (similar errors behaviour can be derived for the other conserved quantities) obtained with the second- to fourth-order entropy-dissipative FD schemes, combined with SSP Runge--Kutta methods of the corresponding orders (SSPRK22, SSPRK33, and SSPRK(5,4)). For the second-order scheme, the observed convergence rate is reduced due to the application of a min-mod limiter to ensure the sign-preserving property. In contrast, the third- and fourth-order schemes achieve the expected convergence rates for both flux formulations. Figure \ref{fig:Cpu_VS_L1_error} shows the CPU time versus the $L_1$ error and indicates that the analytical flux becomes more advantageous on finer grids.
The computational costs of the two flux formulations are comparable on coarse meshes, with the differences depending on the order of the scheme. As the mesh is refined, however, the analytically derived flux becomes consistently more efficient. For the third-order scheme, for example, the computational time on the finest mesh is reduced from $125.0$ to $98.9$, corresponding to a reduction of approximately $21\%$. A similar, though less pronounced, trend is observed for the second- and fourth-order schemes, with reductions of approximately $11\%$ and $7\%$, respectively. Overall, both computational costs increase with mesh refinement and scheme order, while the analytical flux provides a moderate but systematic computational advantage on finer meshes, most notably for the third-order scheme. The complete results are summarized in Table \ref{tab:Tim_comparison}.

\begin{table}[ht]
\centering
\renewcommand{\arraystretch}{1.5}
\resizebox{\textwidth}{!}{%
\begin{tabular}{c|cccc|cccc}
\hline
\text{Order 2}
& \multicolumn{4}{c|}{Analytical flux}
& \multicolumn{4}{c}{Taylor flux } \\
\cline{2-9}
\text{No. of cells} & $L_1$ Error & Rate & $L_2$ Error & Rate
& $L_1$ Error & Rate & $L_2$ Error & Rate \\
\hline

16  & 3.13E-01 & -- & 3.55E-01 & -- & 3.11E-01 & -- & 3.54E-01 & -- \\
32  & 9.75E-02 & 1.682576211 & 1.20E-01 & 1.559193367 & 9.62E-02 & 1.692151823 & 1.20E-01 & 1.563536273 \\
64  & 3.82E-02 & 1.352151518 & 4.68E-02 & 1.364428849 & 3.80E-02 & 1.338707174 & 4.66E-02 & 1.358769389 \\
128 & 1.25E-02 & 1.613675817 & 1.68E-02 & 1.478164552 & 1.24E-02 & 1.611746904 & 1.68E-02 & 1.475061855 \\
256 & 3.49E-03 & 1.838037707 & 5.42E-03 & 1.633002853 & 3.49E-03 & 1.835629549 & 5.41E-03 & 1.632109906 \\
512 & 9.59E-04 & 1.863573094 & 1.72E-03 & 1.651450757 & 9.58E-04 & 1.863508442 & 1.72E-03 & 1.6511005 \\
1024 & 2.58E-04 & 1.8945768 & 5.46E-04 & 1.658832834 & 2.58E-04 & 1.894163179 & 5.46E-04 & 1.658640653 \\
\hline
\end{tabular}%
}
\caption{$L_1$ and $L_2$ errors and order of convergence for $f$ using the entropy-dissipative FD scheme with analytical and Taylor-expanded fluxes.}
\label{tab:O2_f}
\end{table}
\begin{table}[ht]
\centering
\renewcommand{\arraystretch}{1.5}
\resizebox{\textwidth}{!}{%
\begin{tabular}{c|cccc|cccc}
\hline
\text{Order 3}
& \multicolumn{4}{c|}{Analytical flux}
& \multicolumn{4}{c}{Taylor flux} \\
\cline{2-9}
\text{No. of cells} & $L_1$ Error & Rate & $L_2$ Error & Rate
& $L_1$ Error & Rate & $L_2$ Error & Rate \\
\hline

16  & 4.60E-02 & -- & 5.50E-02 & -- & 4.55E-02 & -- & 5.41E-02 & -- \\
32  & 7.83E-03 & 2.556240636 & 9.83E-03 & 2.484238378 & 7.83E-03 & 2.538972332 & 9.79E-03 & 2.466706175 \\
64  & 1.06E-03 & 2.88856732 & 1.37E-03 & 2.842576553 & 1.06E-03 & 2.887647399 & 1.37E-03 & 2.839022094 \\
128 & 1.34E-04 & 2.984356398 & 1.75E-04 & 2.965432659 & 1.34E-04 & 2.98480786 & 1.75E-04 & 2.964354765 \\
256 & 1.67E-05 & 2.997817362 & 2.20E-05 & 2.99383672 & 1.67E-05 & 2.998071528 & 2.20E-05 & 2.993231736 \\
512 & 2.09E-06 & 2.999664921 & 2.75E-06 & 2.998783491 & 2.09E-06 & 2.999796944 & 2.75E-06 & 2.998454591 \\
1024 & 2.61E-07 & 3.000059349 & 3.44E-07 & 2.999690431 & 2.61E-07 & 3.000124763 & 3.44E-07 & 2.999521528 \\
\hline
\end{tabular}%
}
\caption{$L_1$ and $L_2$ errors and order of convergence for $f$ using the entropy-dissipative FD scheme with analytical and Taylor-expanded fluxes.}
\label{tab:O3_f}
\end{table}

\begin{table}[ht]
\centering
\renewcommand{\arraystretch}{1.5}
\resizebox{\textwidth}{!}{%
\begin{tabular}{c|cccc|cccc}
\hline
\text{Order 4}
& \multicolumn{4}{c|}{Analytical flux}
& \multicolumn{4}{c}{Taylor flux} \\
\cline{2-9}
\text{No. of cells} & $L_1$ Error & Rate & $L_2$ Error & Rate
& $L_1$ Error & Rate & $L_2$ Error & Rate \\
\hline

16  & 1.82E-02 & -- & 2.20E-02 & -- & 1.79E-02 & -- & 2.14E-02 & -- \\
32  & 1.73E-03 & 3.388198289 & 2.37E-03 & 3.214505397 & 1.73E-03 & 3.36860498 & 2.35E-03 & 3.184001371 \\
64  & 1.33E-04 & 3.707608471 & 2.37E-04 & 3.32341675 & 1.37E-04 & 3.666133887 & 2.37E-04 & 3.309926718 \\
128 & 1.08E-05 & 3.622514992 & 2.30E-05 & 3.365368482 & 1.10E-05 & 3.635514312 & 2.30E-05 & 3.366566223 \\
256 & 7.49E-07 & 3.847002916 & 1.89E-06 & 3.607112397 & 7.62E-07 & 3.850795627 & 1.89E-06 & 3.607092335 \\
512 & 5.14E-08 & 3.864600627 & 1.55E-07 & 3.599852297 & 5.22E-08 & 3.866261881 & 1.55E-07 & 3.599988746 \\
1024 & 3.43E-09 & 3.90436349 & 1.29E-08 & 3.591605843 & 3.49E-09 & 3.905252299 & 1.29E-08 & 3.591691234 \\
\hline
\end{tabular}%
}
\caption{$L_1$ and $L_2$ errors and order of convergence for $f$ using the entropy-dissipative FD scheme with analytical and Taylor-expanded fluxes.}
\label{tab:O4_f}
\end{table}

\begin{table}[tbhp]
	\footnotesize
	\begin{center}
		\begin{tabular}{c|c|c|c|c|c|c|}
			\hline Number of cells  & \multicolumn{2}{|c}{Order 2} & \multicolumn{2}{|c}{Order 3} & \multicolumn{2}{|c}{Order 4}  \\
			\hline   & Analytical flux  &  Taylor flux &  Analytical flux  &  Taylor flux & Analytical flux  &  Taylor flux \\
			\hline 16 & 5.35E-01 & 5.29E-01 & 5.40E-01 & 5.02E-01 & 7.28E-01 & 8.32E-01 \\
			32 & 6.56E-01 & 6.65E-01 & 7.16E-01 & 7.78E-01 & 1.15E+00 & 1.41E+00 \\
			64 & 8.92E-01 & 7.99E-01 & 1.20E+00 & 1.16E+00 & 2.33E+00 & 2.70E+00 \\
			128 & 1.43E+00 & 1.43E+00 & 2.52E+00 & 2.99E+00 & 6.60E+00 & 8.44E+00 \\
			256 & 3.23E+00 & 3.41E+00 & 7.46E+00 & 8.60E+00 & 2.54E+01 & 2.78E+01 \\
			512 & 9.59E+00 & 1.03E+01 & 2.53E+01 & 3.15E+01 & 9.54E+01 & 1.04E+02 \\
            1024 & 3.17E+01 & 3.58E+01 & 9.89E+01 & 1.25E+02 & 3.46E+02 & 3.74E+02 \\
			\hline
		\end{tabular}
		\caption{Time comparison between analytical flux and Taylor flux (in seconds) using the entropy-dissipative FD scheme.}
		\label{tab:Tim_comparison}
	\end{center}
\end{table}

\subsection{Riemann Problems}\label{test:RP1}

While both approaches show the order of convergence we expect for a smooth solution, we investigate in this subsection how well these methods perform for two adapted Riemann problems.

\subsubsection{A shock-tube type initial data}\label{test:ST}
In the first example, we consider the following initial data
\begin{align}
    \mathbf{U}(x, 0)= \left\{\begin{array}{rcl}
    (1.50, 1.60, 1.60, 2.00)^\top&: &~x<0, \\
    (1.25, 1.15, 2.00, 2.10)^\top &: &~x>0. 
    \end{array} \right.
\end{align}
in the intervall $[-2,12]$. 
The boundary conditions are set as outflow, and the final time is set to $t=1$.\\
\textbf{Entropy dissipative FD: }
In Figure~\eqref{Test:ST}, the results of the entropy-dissipative FD scheme using the analytical fluxes and the Taylor approximated fluxes are presented using $400$ grid points and various orders of accuracy (referred to as $On$, $n=2,3,4$). Nearly no difference can be seen between the analytical derived flux and the Taylor approximated ones except due to the different orders. 

\begin{figure}[!htbp]
	\begin{center}
		\subfigure[$f$]{\includegraphics[width=0.5\textwidth]{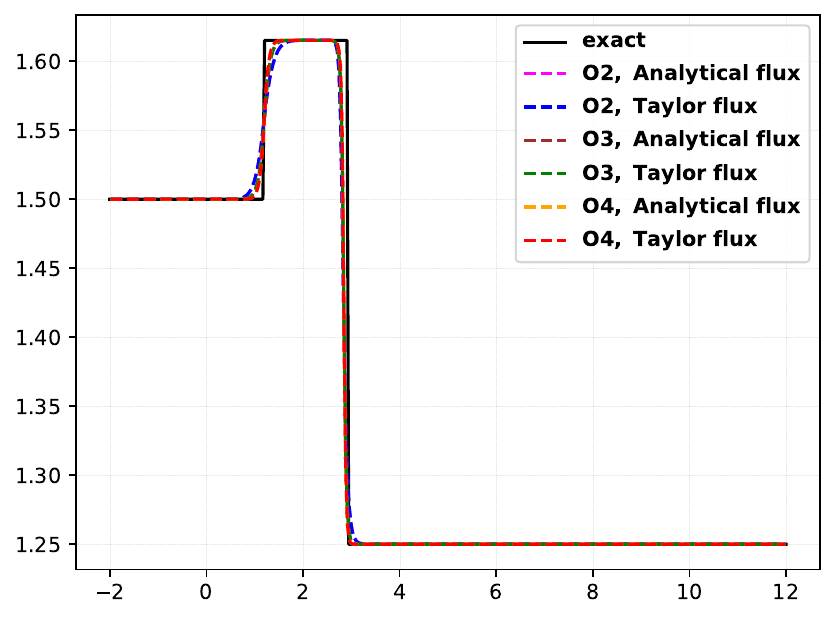}\label{fig:ST_f}}~
             \subfigure[$b$]{\includegraphics[width=0.5\textwidth]{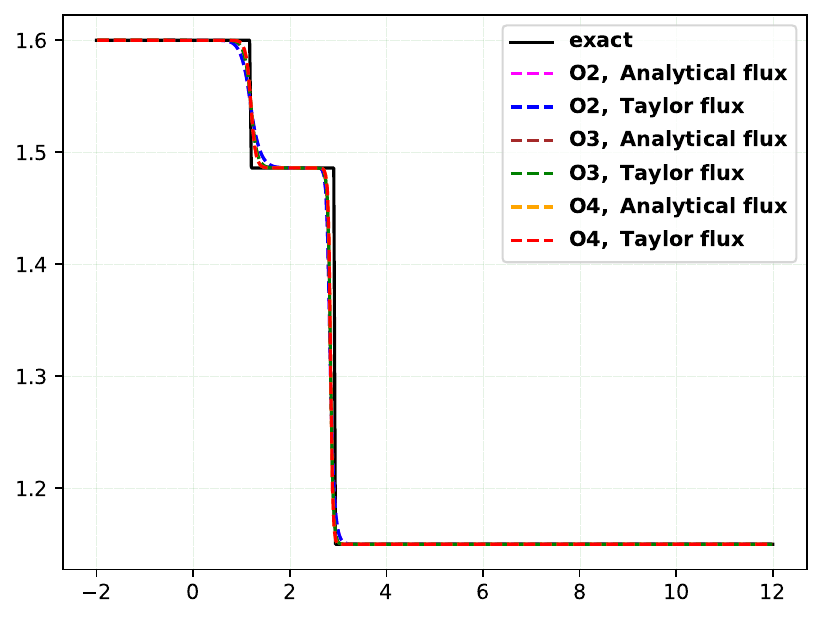}\label{fig:ST_b}}\\
            \subfigure[$g$]{\includegraphics[width=0.5\textwidth]{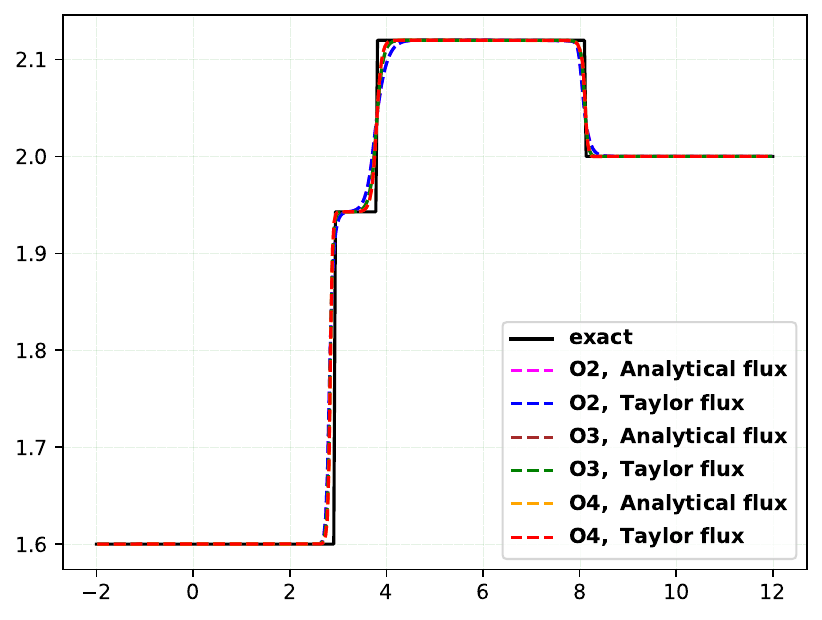}\label{fig:ST_g}}~
            \subfigure[$q$]{\includegraphics[width=0.5\textwidth]{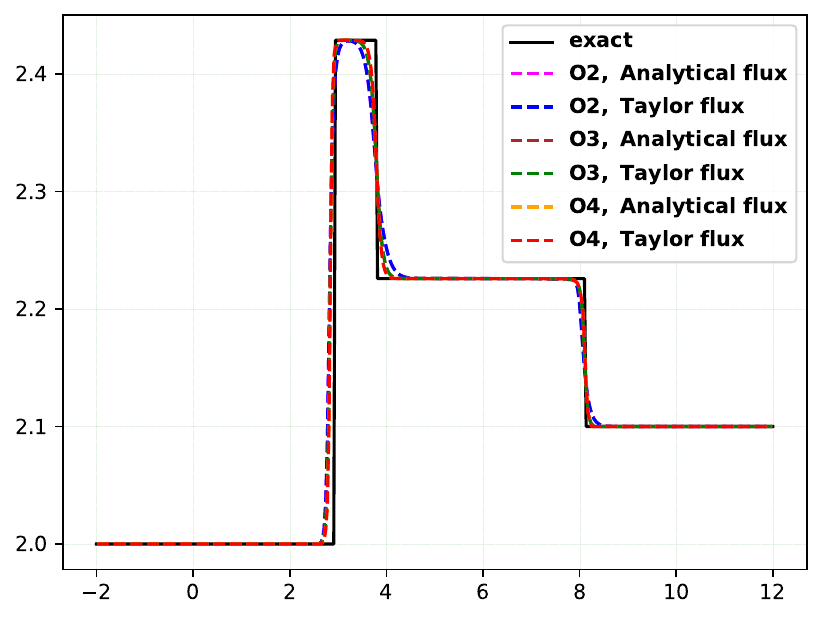}\label{fig:ST_q}}\\
		\caption{\textbf{\nameref{test:ST}}: Plots for $\ote$, $\othe$ and $\ofe$ schemes using 400 cells using Taylor flux.}
        \label{Test:ST}
	\end{center}
\end{figure}
\textbf{DGSEM:} 
We now investigate the performance of the entropy-dissipative DGSE scheme, with particular emphasis on the interplay between flux discretization and shock resolution. The scheme is based on the flux differencing ansatz, combined with our newly developed entropy-conservative numerical fluxes for the volume terms resulting in a skew-symmetric forumulation. To stabilize the solution near discontinuities, we apply local Lax-Friedrichs-type numerical flux as surface flux, which introduces dissipation and improves the robustness.\\
For time integration, we employ the SSPRK(3,3) with a fixed time step $\Delta t = 0.001$. The spatial discretization uses polynomial approximations of degree $p = 2$, resulting in a fully third-order accurate scheme in both space and time. Due to the high-order nature of the method, spurious oscillations due to the Gibbs phenomena emerge near discontinuities. To suppress these oscillations while preserving accuracy in smooth regions, we apply the subcell shock capturing algorithm proposed in~\cite{hennemann2021provably}, adapted to our framework.

The computational domain is discretized using $128$ elements, with corresponding Gauss--Lobatto quadrature points. This setup yields approximately 384 degrees of freedom, providing a fine resolution suitable for capturing sharp gradients and shocks.

\begin{figure}[htbp]
\centering
\includegraphics[width=0.8\textwidth]{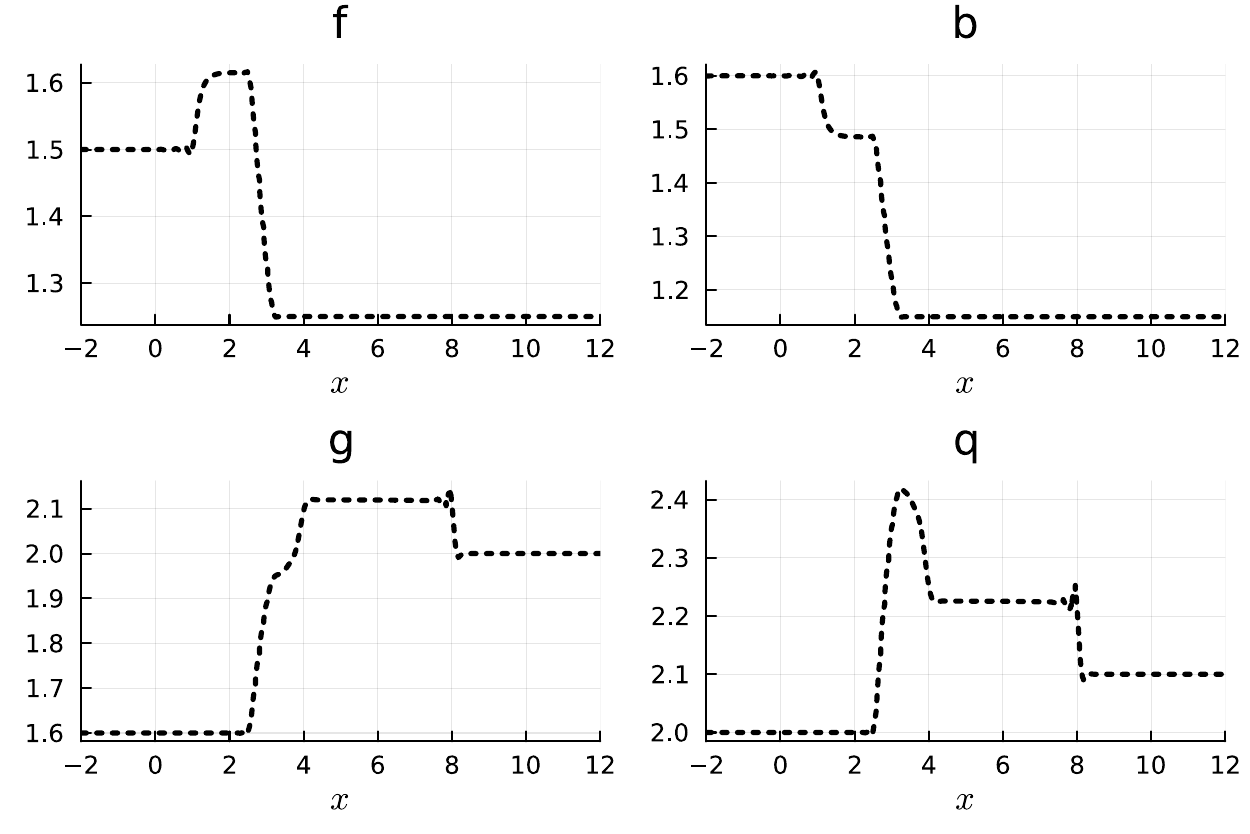}
\label{fig:results_taylor}
\caption{Numerical solutions of the entropy dissipative DSGEM scheme using subcell FV.}
\label{fig:results_DGSEM}
\end{figure}

In Figure~\ref{fig:results_DGSEM}, we present the numerical result obtained by using the Taylor-approximated flux (which, in this case, does not visibly alter to the result obtained by the scheme using the analytical flux)(not shown). Although the solution profiles appear visually similar, a striking difference emerges in the computational cost. The simulation using the Taylor-approximated flux requires approximately $\approx 1.573$ seconds of wall-clock time, whereas the scheme based on the analytical flux completes in only $0.078$ seconds—demonstrating a 20-fold improvement in efficiency.%1.572982, 0.077955

This discrepancy highlights a critical insight: while the Taylor approximation may appear harmless in terms of solution quality, it significantly degrades computational performance due to increased complexity in the flux evaluation solving the system numerically. 

The entropy behavior of both schemes is quantified via the discrete entropy production rate
$
\sum \frac{\partial \eta}{\partial \mathbf{U}} \cdot \mathbf{U}_t,
$
which measures the deviation from entropy conservation. For the Taylor-approximated flux, the entropy change is approximately $-6.21110298 \times 10^{-2}$, and for the analytical flux, it is $-6.21110521 \times 10^{-2}$. Both values are negative, confirming that the scheme is entropy-dissipative. %$-6.21110298 \times 10^{-2}$,$-6.21110521 \times 10^{-2}$

The shock detection and subcell FV approach are performed using a smoothness indicator based on the coefficients of the polynomial expansion. In particular, it will be controlled by algebraic combinations of the conserved variables. After extensive testing of various indicators, we found that $f^4$ provides robust and reliable detection across a range of Riemann problems, including the current test case. This choice not only ensures accurate localization of discontinuities but also maintains stability and accuracy in the subcell finite volume reconstruction even if small oscillation can still be recognized in   $g$ and $q$. Note that examples can be constructed, where $f^4$ does not provide the aforementioned properties.

The presented results exhibit a relatively high level of numerical dissipation, which is expected for a shock-capturing scheme with a subcell FV approach. This dissipation can be reduced by increasing the spatial resolution through grid refinement. In Figure~\ref{fig:results_DGSEM_2}, we demonstrate the solutions for the  variable 
$f$ computed on two different grids: $128$ and $256$ elements. The results demonstrate a clear improvement in resolution with finer grids, particularly in the resolution of sharp gradients. Notably, even the solution obtained by the coarser grid ($128$ cells) shows small oscillations.
In future work, a detailed comparison and analysis of various shock detection techniques and developing other limiting techniques will be conducted to further improve the performance and robustness of the schemes. However, such an investigation lies beyond the scope of this study.

\begin{figure}[htbp]
\centering
\includegraphics[width=0.5\textwidth]{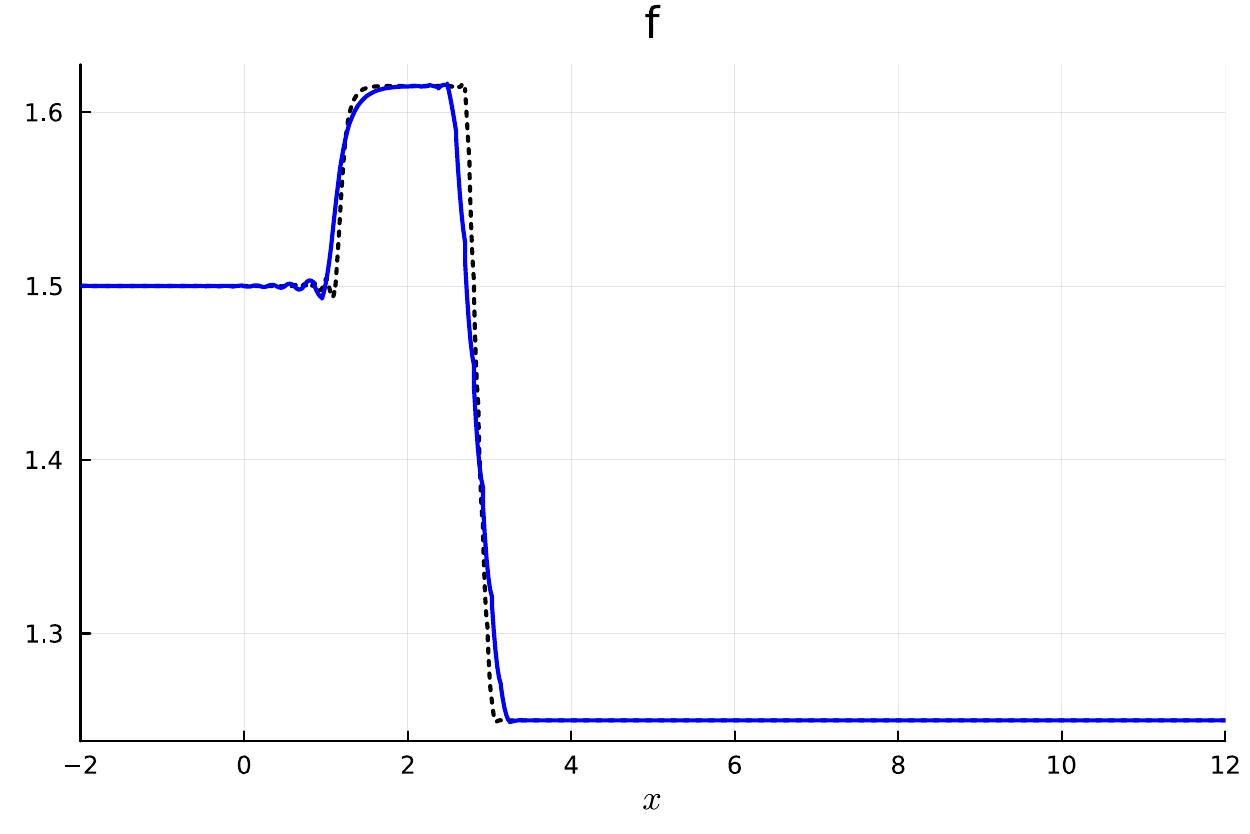}
\label{fig:results_taylor}
\caption{Numerical solutions of the entropy dissipative DSGEM scheme using subcell FV using 128 cells (blue) and 256 cells (black dotted).}
\label{fig:results_DGSEM_2}
\end{figure}

\subsubsection{Riemann problem with negative $b$}\label{test:RP2}

For this test case, we choose our initial data of the form\vspace{-0.2 cm}
\begin{align}
    \mathbf{U}(x, 0)= \left\{\begin{array}{rcl}
    (1.57, -1.15, 2.50, 1.90)^\top&: &~-10\leq x\leq 10, \\
    (1.90,-0.58,2.40,2.30)^\top &: &~10< x\leq 40, 
    \end{array} \right.
\end{align}
end time $t=3.5$ together with outflow conditions. Note that a different state space including negative $b$ is considered. We adapt the flux accordingly for both, Taylor and the analytical one. In Figure~\ref{Test:RP_2}, we present the result using the entropy dissipative FD scheme with the Taylor approximated flux applied with 400 cells whereas for DGSEM, we apply again the flux differencing ansatz together with the subcell FV scheme (same setting as in test \ref{test:ST}, e.g., analytical flux as volume flux and local Lax-Friedrichs flux as surface flux), use 256 cells and present the results for polynomial order $p=2,3$ (together with SSPRK(3,3) and SSPRK(5,4) and $\Delta t=0.001$). The results are presented in 
 Figure \ref{fig:results_DGSEM_3}. Here, we recognized small oscillations in $f$ and smaller in the other variables. 
\begin{figure}[!htbp]
	\begin{center}
		\subfigure[$f$]{\includegraphics[width=0.5\textwidth]{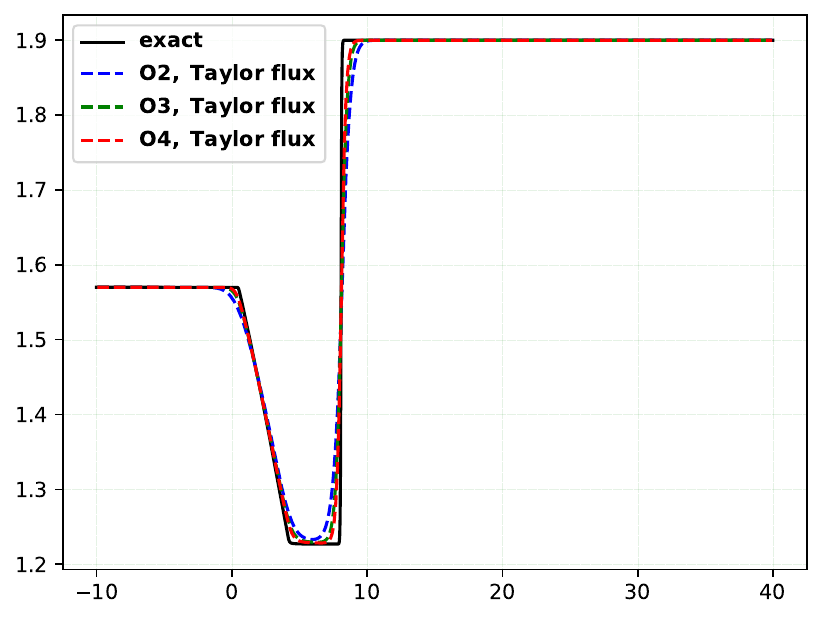}\label{fig:RP2_f}}~
             \subfigure[$b$]{\includegraphics[width=0.5\textwidth]{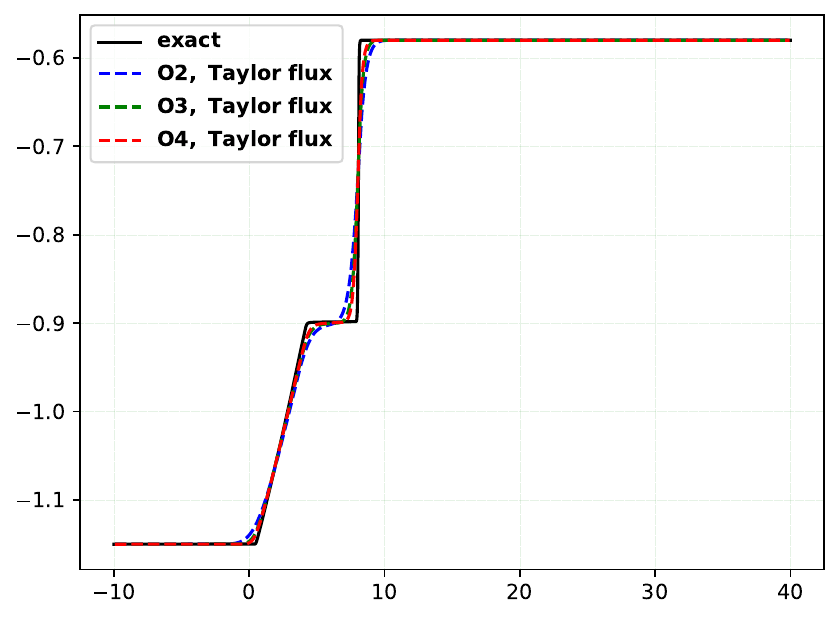}\label{fig:RP2_b}}\\
            \subfigure[$g$]{\includegraphics[width=0.5\textwidth]{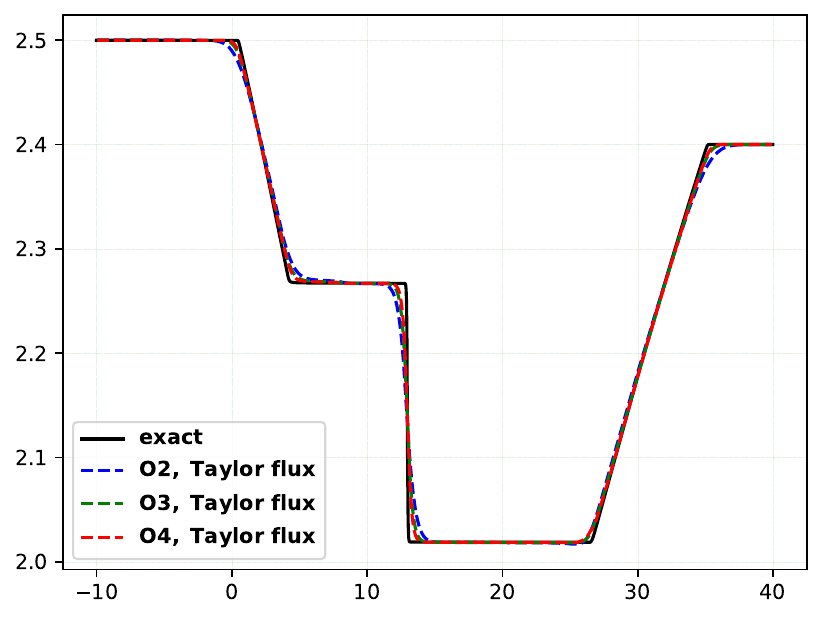}\label{fig:RP2_g}}~
            \subfigure[$q$]{\includegraphics[width=0.5\textwidth]{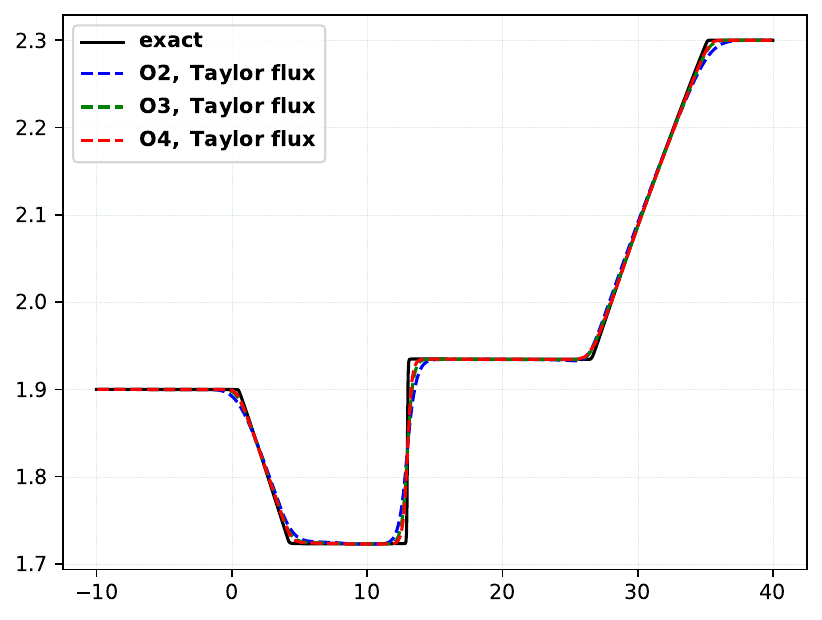}\label{fig:RP2_q}}\\
		\caption{\textbf{\nameref{test:RP2}}: Plots for $\ote$, $\othe$ and $\ofe$ schemes using 400 cells using Taylor flux.}
        \label{Test:RP_2}
	\end{center}
\end{figure}

\begin{figure}[htbp]
\centering
\includegraphics[width=\textwidth]{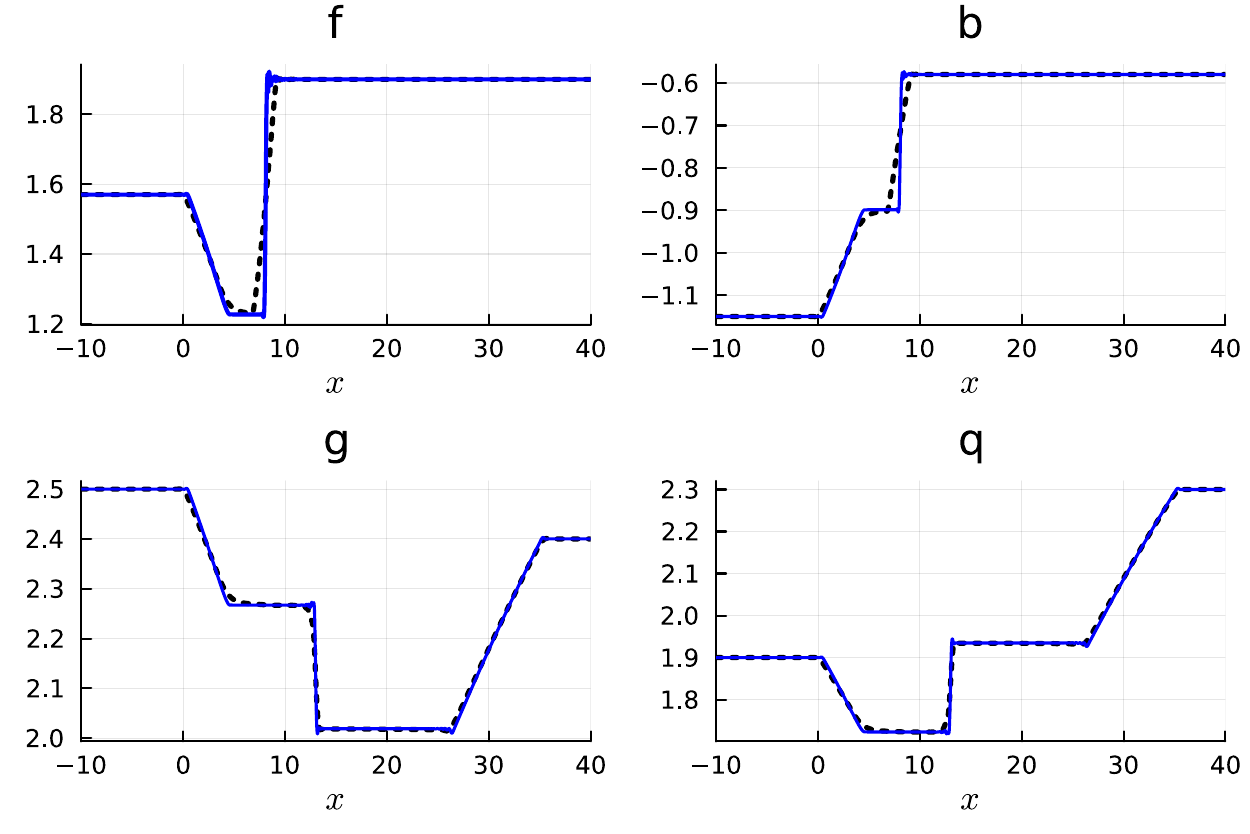}
\label{fig:results_taylor}
\caption{Numerical solutions of the entropy-dissipative DGSEM  using the analytically derived flux and subcell FV, computed on a grid of 256 elements with polynomial degrees 
p=2 (black dotted line) and 
p=3 (blue solid line).}
\label{fig:results_DGSEM_3}
\end{figure}
%--------------------------------------------------------------------
%--------------------------------------------------------------------
\subsection{Two-dimensional problems}

Next, we extend our numerical investigation to two-dimensional problems as introduced in \cite{biswas_2026_Riemann} of the form

\[
\con_t+F(\con)_x+F(\con)_y=0
\]
with $F$ as in \eqref{eq:flux}. 
The initial condition is given by
\begin{align*}
f(x,y,0)&=1.0,\\
b(x,y,0)&=0.5+0.1\phi(r),\\
g(x,y,0)&=2.0,\\
q(x,y,0)&=0.5+0.2\phi(r),\\
\end{align*}
where
$
\phi(r)=\frac{1}{2}\left(1+\tanh(2(2-r))\right)
$
and $r$ denotes the distance to the point $(-1,-1)$  in the domain $\Omega=[-4,4]^2$. Thus, the initial condition consists of a smooth radially symmetric perturbation centered at $(-1,-1)$. \\
In Figure \ref{fig:2d1}, we present the results obtained with the DGSEM using flux differencing and the subcell FV approach. The polynomial order is $p=2$ and the grid is $128^2$. The final time is T=1, and SSPRK(3,3) is used with $\Delta t=0.001$. The surface flux is the local Lax--Friedrichs flux, while the volume flux is the analytically derived flux.
The results are in good agreement with those obtained in \cite{biswas_2026_Riemann} and demonstrate the movement of the circle.
The computation time for this setting is around 21 seconds, whereas the method using the Taylor-approximated flux inside the volume integrals gives a visually indistinguishable result (not shown). However, the computation time on the same machine is around 860 seconds without any improvement to the linear solver. For the Taylor-approximated flux, the resulting linear system has to be solved numerically at each time step, whereas the analytically derived flux provides an explicit formula that can be evaluated directly. We use the in-house linear solver available in Julia for the numerical solution of the Taylor-based system without further solver-specific optimization. We do not claim that this linear system cannot be solved more efficiently using alternative numerical solvers or optimization techniques; rather, the reported computation time reflects the particular implementation used in our comparison.
These results demonstrate that both approaches can be used. However, for two-dimensional simulations, the scheme with the analytical flux is significantly faster, even though it was more complicated to derive.

\begin{figure}
    \centering
    \includegraphics[width=\textwidth]{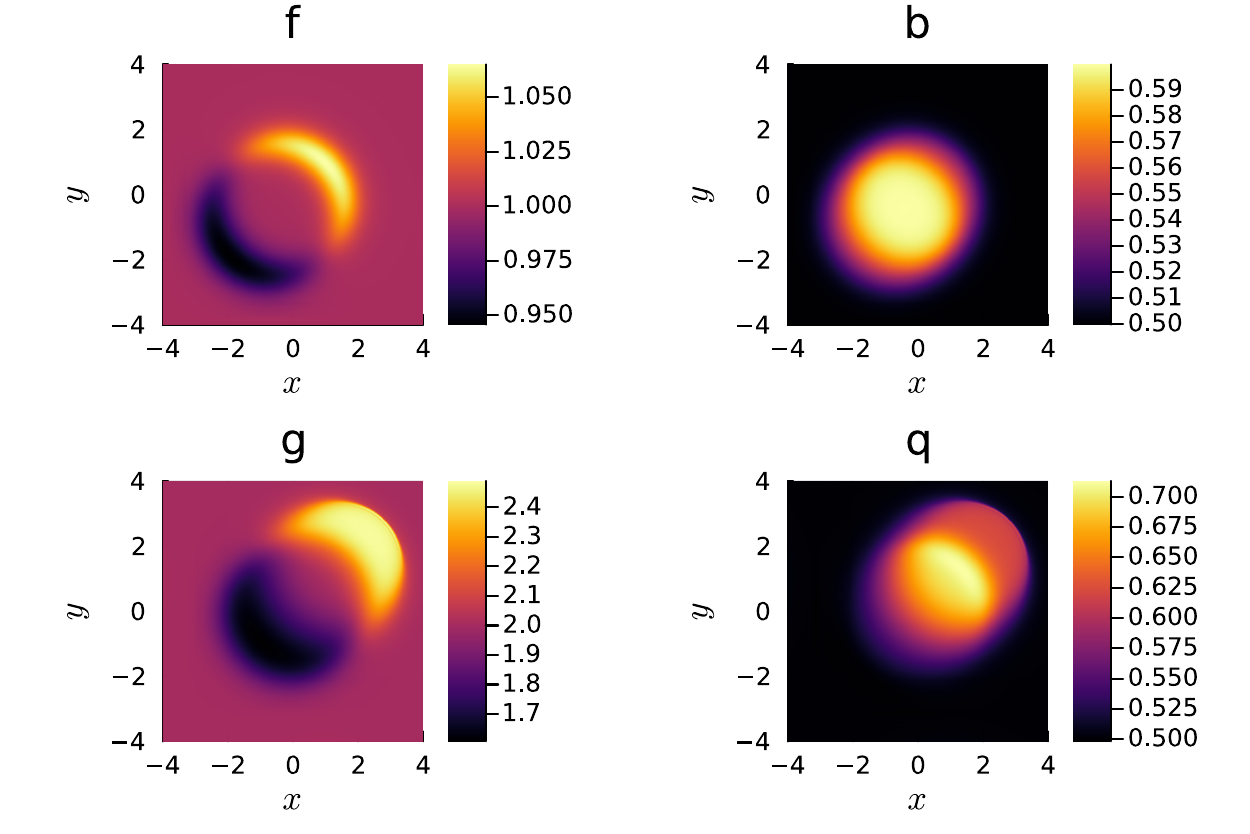}
    \caption{DGSEM with flux differencing and subcell FV shock detection, polynomial order $p=2$, with the analytical flux.}
    \label{fig:2d1}
\end{figure}

For the second test case, we consider a two-dimensional problem on the domain
$
\Omega=[-1,1]^2.
$
The initial condition is given by
\begin{align*}
f(x,y,0)&=1.0,\\
b(x,y,0)&=-\operatorname{sech}^2\left(5(r-R(\theta))\right),\\
g(x,y,0)&=0.8,\\
q(x,y,0)&=\operatorname{sech}^2\left(5(r-R(\theta))\right),
\end{align*}
where
$
r=\sqrt{x^2+y^2},\;
\theta=\operatorname{atan}(y,x),
$
and
$
R(\theta)=0.35+0.03\cos(2\theta).
$

Thus, the initial condition describes a smooth ring-shaped perturbation centered at the origin. The radius of the ring varies periodically with the angle, resulting in a slightly deformed circular shape. The perturbation is localized around $r=R(\theta)$ by the squared hyperbolic secant function, with the parameter $5$ controlling the thickness of the ring. The first and third components remain constant, while the second and fourth components contain perturbations of opposite sign. In particular,
$
b(x,y,0)=-q(x,y,0)\leq 0,
$
with $b=-1$ and $q=1$ at the center of the ring, where $r=R(\theta)$. Away from the ring, both perturbations smoothly approach zero.\\
In Fig.\ref{fig:2d2}, we see the numerical result at $T=0.2$. While $b$ is strictly negative (within the range $(-1,0)$), we see almost identical behavior of $f$ and $g$ with jumps across the $x=y$-line, and $b$ and $q$ being the negative (up to a positive factor) of each other. The jumps (in $f$ and $g$) are resolved very well, even with a rough grid of $128\times 128$ cells.

\begin{figure}
    \centering
    \includegraphics[width=\textwidth]{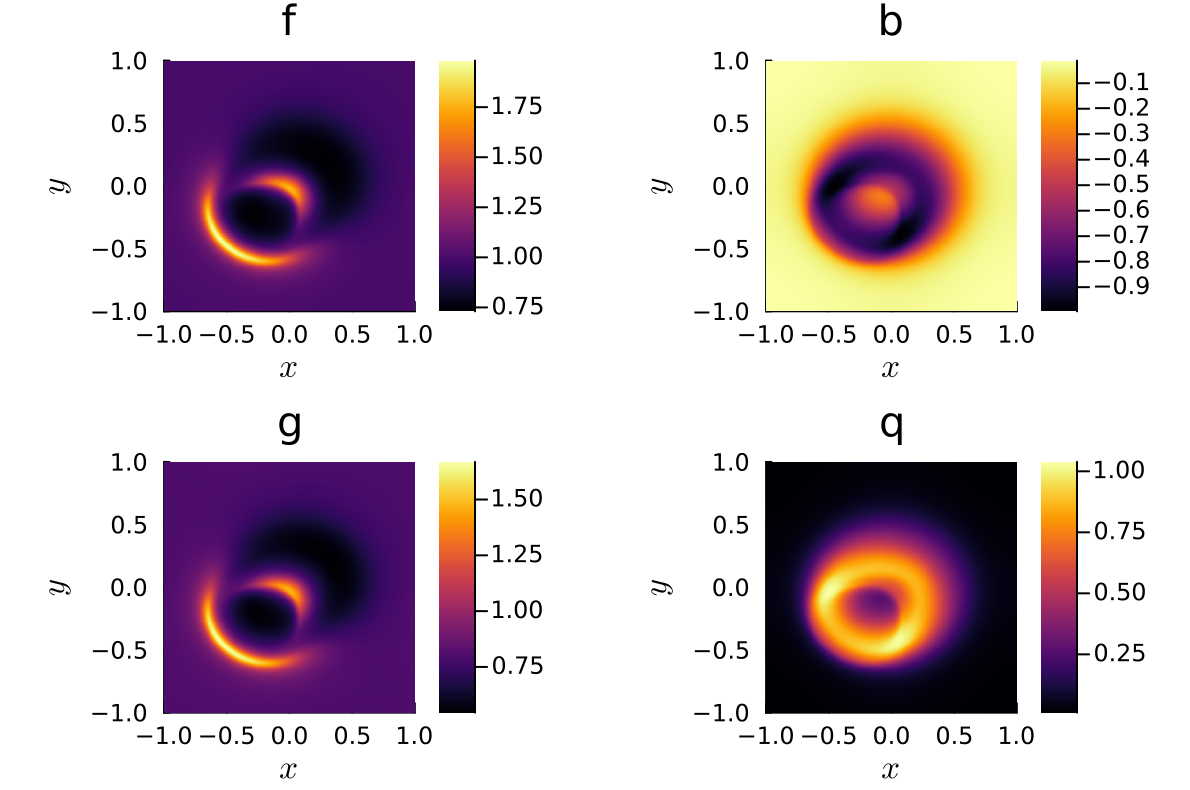}
    \caption{DGSEM with flux differencing and subcell FV shock detection, polynomial order $p=1$, with the analytical flux on a mesh with $128\times 128$ elements.}
    \label{fig:2d2}
\end{figure}

\section{Outlook and Summary}
In this work, we have developed and applied two different entropy-conservative numerical fluxes for a thin-film model. One flux is based on a Taylor approximation, whereas the second one is derived analytically. The Taylor-approximated flux is easier to derive and is based on ideas that are also used in various other settings, e.g., for the Euler equations and shallow water equations. Both fluxes lead to a linear system of equations. While the Taylor-approximated flux has to be solved numerically, the analytical one can be evaluated exactly within each time step of the underlying scheme, resulting in a more efficient approach when applied within numerical schemes.
We have used both approaches and compared them in two different numerical methods. We have employed them in an entropy-dissipative finite-difference setting and within the flux-differencing DGSEM method. The performance was, as expected, similar in terms of accuracy and stability, but an improvement was observed in terms of runtime due to the analytical derivation of the numerical solution. We tested both approaches on smooth problems, Riemann problems in one space dimension, and two-dimensional problems. In particular, for experiments involving shocks, we used additional limiters in the context of the entropy-dissipative finite-difference method, whereas for the DGSEM, we applied the subcell finite-volume technique developed in \cite{hennemann2021provably}.
In the future, we will develop bound-preserving limiters in the spirit of \cite{zbMATH07745801, zbMATH07464528, hennes2021} for the considered model to provably ensure the preservation of the underlying physical laws.

\subsection*{Acknowledgement}
P.Ö. was supported by the German Research Foundation (DFG)  within SPP 2410, project
525866748 (OE 661/5-1) and under the personal grant  520756621 (OE 661/4-1).\\
%--------------------------------------------------------------------
%--------------------------------------------------------------------
%--------------------------------------------------------------------
%--------------------------------------------------------------------
\small
\bibliographystyle{unsrt}  
\bibliography{references}
%--------------------------------------------------------------------
%--------------------------------------------------------------------
%--------------------------------------------------------------------
%--------------------------------------------------------------------
%--------------------------------------------------------------------
%--------------------------------------------------------------------
\normalsize 
\appendix

\section{Details for Taylor expansion}\label{se_Taylor}
Here, we give the concrete idea how the numerical two-point flux is developed from Section \ref{Sec_Tadmor}. Therefore, we 
calculate the jumps of $v_i$ together with the corresponding numerical fluxes using a Taylor expansion and the discrete chain rule. We get:
\begin{equation}\label{eq_Taylor_1}
\begin{aligned}
	\jump{v_1}\cdot F_1 &=[\![f]\!]\left\{ \pm\frac{\ccc{\left(\frac{1}{b f^2}\right)}}{\ccc{b f^2}} 2 \ccc{b}\ccc{f} + \frac{2\ccc{b}~ \ccc{(bf+gq)} ~\ccc{\left(\frac{b}{(bf+gq)^2}\right)}}{\ccc{(bf+gq)^2}}\right\}F_1\\
	&+ [\![b]\!]\left\{ \pm\frac{\ccc{\left(\frac{1}{b f^2}\right)}}{\ccc{b f^2}} \ccc{f^2} - \frac{1 - 2\ccc{f}~ \ccc{(bf+gq)} ~\ccc{\left(\frac{b}{(bf+gq)^2}\right)}}{\ccc{(bf+gq)^2}}\right\}F_1\\
	&+ [\![g]\!]\left\{ \frac{2\ccc{q}~ \ccc{(bf+gq)} ~\ccc{\left(\frac{b}{(bf+gq)^2}\right)}}{\ccc{(bf+gq)^2}}\right\}F_1 + [\![q]\!]\left\{ \frac{2\ccc{g}~ \ccc{(bf+gq)} ~\ccc{\left(\frac{b}{(bf+gq)^2}\right)}}{\ccc{(bf+gq)^2}}\right\}F_1,\\
	\jump{v_2}\cdot F_2 &=[\![f]\!]\left\{ \pm\frac{\ccc{\left(\frac{1}{b^2 f}\right)}}{\ccc{b^2 f}} \ccc{b^2} - \frac{1- 2\ccc{b}~ \ccc{(bf+gq)} ~\ccc{\left(\frac{f}{(bf+gq)^2}\right)}}{\ccc{(bf+gq)^2}}\right\}F_2\\
	&+ [\![b]\!]\left\{ \pm\frac{\ccc{\left(\frac{1}{b^2 f}\right)}}{\ccc{b^2 f}} 2 \ccc{b}\ccc{f} + \frac{2\ccc{f}~ \ccc{(bf+gq)} ~\ccc{\left(\frac{f}{(bf+gq)^2}\right)}}{\ccc{(bf+gq)^2}}\right\}F_2\\
	&+ [\![g]\!]\left\{ \frac{2\ccc{q}~ \ccc{(bf+gq)} ~\ccc{\left(\frac{f}{(bf+gq)^2}\right)}}{\ccc{(bf+gq)^2}}\right\}F_2 + [\![q]\!]\left\{ \frac{2\ccc{g}~ \ccc{(bf+gq)} ~\ccc{\left(\frac{f}{(bf+gq)^2}\right)}}{\ccc{(bf+gq)^2}}\right\}F_2,
\\
	\jump{v_3}\cdot F_3 &=[\![f]\!]\left\{ \frac{ 2\ccc{b}~ \ccc{(bf+gq)} ~\ccc{\left(\frac{q}{(bf+gq)^2}\right)}}{\ccc{(bf+gq)^2}}\right\}F_3 + [\![b]\!]\left\{ \frac{2\ccc{f}~ \ccc{(bf+gq)} ~\ccc{\left(\frac{q}{(bf+gq)^2}\right)}}{\ccc{(bf+gq)^2}}\right\}F_3\\
	&+ [\![g]\!]\left\{ \frac{2\ccc{q}~ \ccc{(bf+gq)} ~\ccc{\left(\frac{q}{(bf+gq)^2}\right)}}{\ccc{(bf+gq)^2}}\right\}F_3 \\
    &+ [\![q]\!]\left\{ -\frac{1- 2\ccc{g}~ \ccc{(bf+gq)} ~\ccc{\left(\frac{q}{(bf+gq)^2}\right)}}{\ccc{(bf+gq)^2}}\right\}F_3,
\\
	\jump{v_4}\cdot F_4 &=[\![f]\!]\left\{ \frac{ 2\ccc{b}~ \ccc{(bf+gq)} ~\ccc{\left(\frac{g}{(bf+gq)^2}\right)}}{\ccc{(bf+gq)^2}}\right\}F_4 + [\![b]\!]\left\{ \frac{2\ccc{f}~ \ccc{(bf+gq)} ~\ccc{\left(\frac{g}{(bf+gq)^2}\right)}}{\ccc{(bf+gq)^2}}\right\}F_4\\
	&+ [\![g]\!]\left\{ -\frac{1- 2\ccc{q}~ \ccc{(bf+gq)} ~\ccc{\left(\frac{g}{(bf+gq)^2}\right)}}{\ccc{(bf+gq)^2}}\right\}F_4 \\
    &+ [\![q]\!]\left\{\frac{2\ccc{g}~ \ccc{(bf+gq)} ~\ccc{\left(\frac{g}{(bf+gq)^2}\right)}}{\ccc{(bf+gq)^2}}\right\}F_4.
\end{aligned}
\end{equation}
%==============================================================================
%==============================================================================
Further, we simplify $[\![ \Psi]\!]$, and obtain 
\begin{align}
	[\![ \Psi]\!]&= \frac{\ccc{f}[\![b]\!] + \ccc{b}[\![f]\!] - \ccc{\left(\frac{bf}{(bf+gq)}\right)}\left\{\ccc{b}[\![f]\!] + \ccc{f}[\![b]\!] + \ccc{g}[\![q]\!] + \ccc{q}[\![g]\!]\right\}}{2~\ccc{(bf+gq)}}\nonumber\\
	&+ \frac{3}{2}\left\{\pm\frac{[\![b]\!]}{\ccc{b}_{\log}} \pm \frac{[\![f]\!]}{\ccc{f}_{\log}}  + \frac{\ccc{b}[\![f]\!] + \ccc{f}[\![b]\!] + \ccc{g}[\![q]\!] + \ccc{q}[\![g]\!]}{  \ccc{bf+gq}_{\log} }\right\}\nonumber\\
	[\![ \Psi]\!]&= [\![f]\!]\left\{ \pm\frac{3}{2\ccc{f}_{\text{log}}} 
    + \frac{3\ccc{b}}{2\ccc{bf+gq}_{\text{log}}} + \frac{\ccc{b} - \ccc{b}~ \ccc{\left(\frac{bf}{(bf+gq)}\right)}}{2~\ccc{(bf+gq)}}\right\} \nonumber \\& 
    + [\![b]\!]\left\{ \pm\frac{3}{2\ccc{b}_{\text{log}}} + \frac{3\ccc{f}}{2\ccc{bf+gq}_{\text{log}}} + \frac{\ccc{f} - \ccc{f}~ \ccc{\left(\frac{bf}{(bf+gq)}\right)}}{2~\ccc{(bf+gq)}}\right\} \nonumber\\
	&+ [\![g]\!]\left\{ \frac{3\ccc{q}}{2\ccc{bf+gq}_{\text{log}}} - \frac{\ccc{q}~ \ccc{\left(\frac{bf}{(bf+gq)}\right)}}{2~\ccc{(bf+gq)}}\right\}  + [\![q]\!]\left\{ \frac{3\ccc{g}}{2\ccc{bf+gq}_{\text{log}}} - \frac{\ccc{g}~ \ccc{\left(\frac{bf}{(bf+gq)}\right)}}{2~\ccc{(bf+gq)}}\right\}.\label{Jump_e_p}
\end{align}

Equating the coefficients of $\jump{f}$, $\jump{b}$, $\jump{g}$, and
$\jump{q}$ in \eqref{eq_Taylor_1} and \eqref{Jump_e_p},
we obtain the following $4\times4$ linear system:
\begin{equation*}
\mathbf{A} F^{\mathrm{num}}=\mathbf{b}. 
\end{equation*}
The coefficient matrix and the right-hand side are given explicitly below. It is 
\begin{equation*}
\mathbf{A}
=
\begin{pmatrix}
A_{11} & A_{12} & A_{13} & A_{14}\\
A_{21} & A_{22} & A_{23} & A_{24}\\
A_{31} & A_{32} & A_{33} & A_{34}\\
A_{41} & A_{42} & A_{43} & A_{44}
\end{pmatrix},
\end{equation*}
with
\begin{align*}
A_{11}
&=
\pm
\frac{\ccc{\left(\frac{1}{bf^2}\right)}}{\ccc{bf^2}}
\,2\ccc{b}\ccc{f}
+
\frac{
2\ccc{b}\,\ccc{(bf+gq)}
\,\ccc{\left(\frac{b}{(bf+gq)^2}\right)}
}{
\ccc{(bf+gq)^2}
},
\\
A_{12}
&=
\pm
\frac{\ccc{\left(\frac{1}{b^2f}\right)}}{\ccc{b^2f}}
\,\ccc{b^2}
-
\frac{
1-
2\ccc{b}\,\ccc{(bf+gq)}
\,\ccc{\left(\frac{f}{(bf+gq)^2}\right)}
}{
\ccc{(bf+gq)^2}
},
\\
A_{13}
&=
\frac{
2\ccc{b}\,\ccc{(bf+gq)}
\,\ccc{\left(\frac{q}{(bf+gq)^2}\right)}
}{
\ccc{(bf+gq)^2}
},
\qquad
A_{14}
=
\frac{
2\ccc{b}\,\ccc{(bf+gq)}
\,\ccc{\left(\frac{g}{(bf+gq)^2}\right)}
}{
\ccc{(bf+gq)^2}
},
\end{align*}

\begin{align*}
A_{21}
&=
\pm
\frac{\ccc{\left(\frac{1}{bf^2}\right)}}{\ccc{bf^2}}
\,\ccc{f^2}
-
\frac{
1-
2\ccc{f}\,\ccc{(bf+gq)}
\,\ccc{\left(\frac{b}{(bf+gq)^2}\right)}
}{
\ccc{(bf+gq)^2}
},
\\
A_{22}
&=
\pm
\frac{\ccc{\left(\frac{1}{b^2f}\right)}}{\ccc{b^2f}}
\,2\ccc{b}\ccc{f}
+
\frac{
2\ccc{f}\,\ccc{(bf+gq)}
\,\ccc{\left(\frac{f}{(bf+gq)^2}\right)}
}{
\ccc{(bf+gq)^2}
},
\\
A_{23}
&=
\frac{
2\ccc{f}\,\ccc{(bf+gq)}
\,\ccc{\left(\frac{q}{(bf+gq)^2}\right)}
}{
\ccc{(bf+gq)^2}
},
\qquad
A_{24}
=
\frac{
2\ccc{f}\,\ccc{(bf+gq)}
\,\ccc{\left(\frac{g}{(bf+gq)^2}\right)}
}{
\ccc{(bf+gq)^2}
},
\end{align*}

\begin{align*}
A_{31}
&=
\frac{
2\ccc{q}\,\ccc{(bf+gq)}
\,\ccc{\left(\frac{b}{(bf+gq)^2}\right)}
}{
\ccc{(bf+gq)^2}
},
\qquad
A_{32}
=
\frac{
2\ccc{q}\,\ccc{(bf+gq)}
\,\ccc{\left(\frac{f}{(bf+gq)^2}\right)}
}{
\ccc{(bf+gq)^2}
},
\\
A_{33}
&=
\frac{
2\ccc{q}\,\ccc{(bf+gq)}
\,\ccc{\left(\frac{q}{(bf+gq)^2}\right)}
}{
\ccc{(bf+gq)^2}
},
\qquad
A_{34}
=
-
\frac{
1-
2\ccc{q}\,\ccc{(bf+gq)}
\,\ccc{\left(\frac{g}{(bf+gq)^2}\right)}
}{
\ccc{(bf+gq)^2}
},
\end{align*}

\begin{align*}
A_{41}
&=
\frac{
2\ccc{g}\,\ccc{(bf+gq)}
\,\ccc{\left(\frac{b}{(bf+gq)^2}\right)}
}{
\ccc{(bf+gq)^2}
},
\qquad
A_{42}
=
\frac{
2\ccc{g}\,\ccc{(bf+gq)}
\,\ccc{\left(\frac{f}{(bf+gq)^2}\right)}
}{
\ccc{(bf+gq)^2}
},
\\
A_{43}
&=
-
\frac{
1-
2\ccc{g}\,\ccc{(bf+gq)}
\,\ccc{\left(\frac{q}{(bf+gq)^2}\right)}
}{
\ccc{(bf+gq)^2}
},
\qquad
A_{44}
=
\frac{
2\ccc{g}\,\ccc{(bf+gq)}
\,\ccc{\left(\frac{g}{(bf+gq)^2}\right)}
}{
\ccc{(bf+gq)^2}
}.
\end{align*}
and the right hand side is

\begin{align*}
b_1
&=
\pm\frac{3}{2\ccc{f}_{\log}}
+
\frac{3\ccc{b}}{2\ccc{(bf+gq)}_{\log}}
+
\frac{
\ccc{b}
-
\ccc{b}\,\ccc{\left(\frac{bf}{bf+gq}\right)}
}{
2\ccc{(bf+gq)}
},
\\
b_2
&=
\pm\frac{3}{2\ccc{b}_{\log}}
+
\frac{3\ccc{f}}{2\ccc{(bf+gq)}_{\log}}
+
\frac{
\ccc{f}
-
\ccc{f}\,\ccc{\left(\frac{bf}{bf+gq}\right)}
}{
2\ccc{(bf+gq)}
},
\\
b_3
&=
\frac{3\ccc{q}}{2\ccc{(bf+gq)}_{\log}}
-
\frac{
\ccc{q}\,\ccc{\left(\frac{bf}{bf+gq}\right)}
}{
2\ccc{(bf+gq)}
},
\\
b_4
&=
\frac{3\ccc{g}}{2\ccc{(bf+gq)}_{\log}}
-
\frac{
\ccc{g}\,\ccc{\left(\frac{bf}{bf+gq}\right)}
}{
2\ccc{(bf+gq)}
}.
\end{align*}
 We provide the exact form in the corresponding repository where it can be downloaded \cite{code}.

\section{Detail for analytical derivation} 
Here, we give now the details of the calculation to derive the analytical expression of the entropy-conservative two-point flux from Section \ref{sec_Analytical}. We first give the entropy potential in terms of our abbreviations yielding to~\eqref{eq_potential_model}. We have
\label{sec_analytical}
\begin{align*}
\Psi&=\begin{pmatrix}-u_2(A^{-2}+B^{-2})\\-u_1(A^{-2}+B^{-2})\\-u_4B^{-2}\\-u_3B^{-2}\end{pmatrix}\cdot\begin{pmatrix}\frac{1}{2}u_1A\\ \frac{1}{2}u_2A\\\frac{1}{2}u_3(A+B)\\\frac{1}{2}u_4(A+B)\end{pmatrix}+\frac{3}{2}\ln(AB)+\frac{1}{2}AB^{-1}\\
&=-\frac{1}{2}\big(A^2(A^{-2}+B^{-2})+A^2(A^{-2}+B^{-2})+2u_3u_4(A+B)B^{-2}-3\ln(AB)-AB^{-1}\big)\\
&=-\frac{1}{2}\big(2+2A^2B^{-2}+2(B-A)(B+A)B^{-2}-3\ln(AB)-AB^{-1}\big)\\
&=-\frac{1}{2}\big(4-AB^{-1}-3\ln(AB)\big).
\end{align*}
Before starting with the calcuation, we note that  for $n\in\mathbb{N}$
\[
0=\bbb{1}=\bbb{a^na^{-n}}=\bbb{a^n}\ccc{a^{-n}}+\ccc{a^n}\bbb{a^{-n}},
\]
and therefore obtain 
\begin{equation}
\bbb{a^{-n}}=-\bbb{a^n}\frac{\ccc{a^{-n}}}{\ccc{a^n}}.\label{eq:bbv-n}
\end{equation}
To find the numerical flux, we use \eqref{eq:tadmorsCondition} and get
\begin{align*}
0=&\bbb{\evar(\con)}\cdot F^\text{num}-\bbb{\Psi}=\begin{pmatrix}\bbb{-u_2(A^{-2}+B^{-2})}\\\bbb{-u_1(A^{-2}+B^{-2})}\\\bbb{-u_4B^{-2}}\\\bbb{-u_3B^{-2}}\end{pmatrix}\cdot\begin{pmatrix}F_1\\F_2\\F_3\\F_4\end{pmatrix}+\frac{1}{2}\bbb{4-AB^{-1}-3\ln(AB)}\\
=&-F_1\bbb{u_2(A^{-2}+B^{-2})}-F_2\bbb{u_1(A^{-2}+B^{-2})}-F_3\bbb{u_4B^{-2}}-F_4\bbb{u_3B^{-2}}\\
&-\frac{1}{2}\big(\bbb{AB^{-1}}+3\bbb{\ln(AB)}\big).
\end{align*}
The first four components are of the form $F_i\bbb{u_jX^{-2}}$, where $i,j\in\{1,2,3,4\}$ and $X\in\{A,B\}$. First, we find,
\[
\bbb{u_iX^{-2}}=\bbb{u_i}\ccc{X^{-2}}+\ccc{u_i}\bbb{X^{-2}},
\]
then we expand the last term using \eqref{eq:bbv-n} via
\[
\bbb{X^{-2}}=-\frac{\ccc{X^{-2}}}{\ccc{X^2}}\bbb{X^2},
\]
and calculate separately
\[
\bbb{A^2}=2\bbb{A}\ccc{A}=2\bbb{u_1u_2}\ccc{A}=2\ccc{A}\big(\bbb{u_1}\ccc{u_2}+\bbb{u_2}\ccc{u_1}\big)
\]
and
\begin{align*}
\bbb{B^2}&=2\bbb{B}\ccc{B}=2\ccc{B}\big(\bbb{u_1u_2}+\bbb{u_3u_4}\big)\\
&=2\ccc{B}\big(\bbb{u_1}\ccc{u_2}+\bbb{u_2}\ccc{u_1}+\bbb{u_3}\ccc{u_4}+\bbb{u_4}\ccc{u_3}\big).
\end{align*}
Further, we have
\[
\bbb{AB^{-1}}=\bbb{A}\ccc{B^{-1}}+\bbb{B^{-1}}\ccc{A}=\bbb{A}\ccc{B^{-1}}-\bbb{B}\frac{\ccc{B^{-1}}}{\ccc{B}}\ccc{A},
\]
%and
%\[
%\bbb{\log(u)}=\frac{\bbb{u}}{\ccc{u}_{\log}},
%\]
%where
%\[
%\ccc{u}_{\log}:=\frac{u^+-u^-}{\log(u)^+-\log(u)^-}.
%\]
and combining all, leads to
\begin{align*}
0&=-F_1\Big(\bbb{u_2}\ccc{A^{-2}+B^{-2}}+\ccc{u_2}\big(-\frac{\ccc{A^{-2}}}{\ccc{A^2}}2\ccc{A}(\bbb{u_1}\ccc{u_2}+\bbb{u_2}\ccc{u_1})\big)\\
&+\ccc{u_2}\big(-\frac{\ccc{B^{-2}}}{\ccc{B^2}}2\ccc{B}(\bbb{u_1}\ccc{u_2}+\bbb{u_2}\ccc{u_1}+\bbb{u_3}\ccc{u_4}+\bbb{u_4}\ccc{u_3})\big)\Big)\\
&-F_2\Big(\bbb{u_1}\ccc{A^{-2}+B^{-2}}+\ccc{u_1}\big(-\frac{\ccc{A^{-2}}}{\ccc{A^2}}2\ccc{A}(\bbb{u_1}\ccc{u_2}+\bbb{u_2}\ccc{u_1})\big)\\
&+\ccc{u_1}\big(-\frac{\ccc{B^{-2}}}{\ccc{B^2}}2\ccc{B}(\bbb{u_1}\ccc{u_2}+\bbb{u_2}\ccc{u_1}+\bbb{u_3}\ccc{u_4}+\bbb{u_4}\ccc{u_3})\big)\Big)\\
&-F_3\Big(\bbb{u_4}\ccc{B^{-2}}+\ccc{u_4}\big(-\frac{\ccc{B^{-2}}}{\ccc{B^2}}2\ccc{B}(\bbb{u_1}\ccc{u_2}+\bbb{u_2}\ccc{u_1}+\bbb{u_3}\ccc{u_4}+\bbb{u_4}\ccc{u_3})\big)\Big)\\
&-F_4\Big(\bbb{u_3}\ccc{B^{-2}}+\ccc{u_3}\big(-\frac{\ccc{B^{-2}}}{\ccc{B^2}}2\ccc{B}(\bbb{u_1}\ccc{u_2}+\bbb{u_2}\ccc{u_1}+\bbb{u_3}\ccc{u_4}+\bbb{u_4}\ccc{u_3})\big)\Big)\\
&+\frac{1}{2}\ccc{B^{-1}}(\bbb{u_1}\ccc{u_2}\ccc{u_1}\bbb{u_2})-\frac{1}{2}\ccc{A}\frac{\ccc{B^{-1}}}{\ccc{B}}(\bbb{u_1}\ccc{u_2}+\ccc{u_1}\bbb{u_2}+\bbb{u_3}\ccc{u_4}+\ccc{u_3}\bbb{u_4})\\
&-\frac{3}{2}\ccc{AB}_{\log}^{-1}\big(\ccc{B}(\bbb{u_1}\ccc{u_2}+\bbb{u_2}\ccc{u_1})+\ccc{A}(\bbb{u_1}\ccc{u_2}+\bbb{u_2}\ccc{u_1}+\bbb{u_3}\ccc{u_4}+\bbb{u_4}\ccc{u_3})\big).
\end{align*}
Defining
\[
\ccc{\tilde{C}}:=\ccc{C}\frac{\ccc{C^{-2}}}{\ccc{C^2}},\quad C\in\{A,B\},
\]
we find
\begin{align*}
0&=\bbb{u_1}\Big(2F_1\ccc{u_2}^2\ccc{\tilde{A}}+2f_1\ccc{u_2}^2\ccc{\tilde{B}}-F_2\ccc{A^{-2}+B^{-2}}\\
&\quad\quad+2F_2\ccc{u_1}\ccc{u_2}\ccc{\tilde{A}}+2F_2\ccc{u_1}\ccc{u_2}\ccc{\tilde{B}}+2F_3\ccc{u_2}\ccc{u_4}\ccc{\tilde{B}}+2F_4\ccc{u_3}\ccc{u_2}\ccc{\tilde{B}}\\
&\quad\quad+\frac{1}{2}\ccc{u_2}\ccc{B^{-1}}-\frac{1}{2}\ccc{u_2}\ccc{A}\frac{\ccc{B^{-1}}}{\ccc{B}}-\frac{3}{2}\ccc{u_2}(\ccc{A}+\ccc{B})\ccc{AB}_{\log}^{-1}\Big)\\
&+\bbb{u_2}\Big(-f_1\ccc{A^{-2}+B^{-2}}+2f_1\ccc{u_1}\ccc{u_2}\ccc{\tilde{A}}+2f_1\ccc{u_1}\ccc{u_2}\ccc{\tilde{B}}\\
&\quad\quad+2F_2\ccc{u_1}^2\ccc{\tilde{A}}+2\ccc{u_1}^2\ccc{\tilde{B}}+2F_3\ccc{u_1}\ccc{u_4}\ccc{\tilde{B}}+2F_4\ccc{u_1}\ccc{u_3}\ccc{\tilde{B}}\\
&\quad\quad+\frac{1}{2}\ccc{u_1}\ccc{B^{-1}}-\frac{1}{2}\ccc{u_1}\ccc{A}\frac{\ccc{B^{-1}}}{\ccc{B}}-\frac{3}{2}\ccc{u_1}(\ccc{A}+\ccc{B})\ccc{AB}_{\text{log}}^{-1}\Big)\\
&+\bbb{u_3}\Big(2f_1\ccc{u_2}\ccc{u_4}\ccc{\tilde{B}}+2F_2\ccc{u_1}\ccc{u_4}\ccc{\tilde{B}}+2F_3\ccc{u_4}^2\ccc{\tilde{B}}\\
&\quad\quad+F_4(2\ccc{u_3}\ccc{u_4}\ccc{\tilde{B}}-\ccc{B^{-2}})-\frac{1}{2}\ccc{u_4}\ccc{A}\frac{\ccc{B^{-1}}}{\ccc{B}}-\frac{3}{2}\ccc{u_4}\ccc{A}\ccc{AB}_{\log}^{-1}\Big)\\
&+\bbb{u_4}\Big(2f_1\ccc{u_2}\ccc{u_3}\ccc{\tilde{B}}+2F_2\ccc{u_1}\ccc{u_3}\ccc{\tilde{B}}+F_3(-\ccc{B^{-2}}+2\ccc{u_3}\ccc{u_4}\ccc{\tilde{B}})\\
&\quad\quad+2F_4\ccc{u_3}^2\ccc{\tilde{B}}-\frac{1}{2}\ccc{u_3}\ccc{A}\frac{\ccc{B^{-1}}}{\ccc{B}}-\frac{3}{2}\ccc{u_3}\ccc{A}\ccc{AB}_{\log}^{-1}\Big).
\end{align*}
From $\alpha\bbb{u_1}+\beta\bbb{u_2}+\gamma\bbb{u_3}+\delta\bbb{u_4}=0$, it follows that $\alpha=\beta=\gamma=\delta=0$, leading to the linear system  \eqref{eq_system}.

\section{Entropy-scaled right eigenvectors for two-layer thin film flow model}\label{Eiegn_vector_for_con}
Here, we will compute the entropy-scaled right eigenvectors in $x$-direction using the Barth scaling procedure~\cite{barth1999numerical}. We first construct the right eigenvector matrix $R^x$ corresponding to the Jacobian matrix $\frac{\p F(u)}{\p \con}$, which is given below,
\begin{align*}
	\begin{pmatrix}
		-\frac{f}{b} & \frac{fb-3gq}{4qb} & 0 & 0\\
        1 & \frac{fb-3gq}{4qf} & 0 & 0\\
        0 & \frac{g}{q} & -\frac{g}{q} & \frac{g}{q}\\
        0 & 1 & 1 & 1
	\end{pmatrix}.
\end{align*}
We want to determine the scaling matrix $T^x$ that will allow the scaled right-eigenvector matrix $\tilde{R}^x=R^x T^x$ to satisfy
\begin{equation}
	\frac{\p \con}{\p \evar} = \tilde{R^x} \tilde{R^x}^{\top},
\end{equation}
where the entropy variable $\evar$ is defined in~\eqref{eq:envar}. Now we follow~\cite{barth1999numerical} and define the matrix 
\begin{equation*}
	\mathcal{Y}^x = \left(R^x\right)^{-1} \frac{\p \con}{\p \evar}\left(R^x\right)^{-\top}.
\end{equation*}
For $b>0$, this yields
\begin{align*}
	\mathcal{Y}^x = \begin{pmatrix}
		\frac{f b^3\left(bf+gq\right)^2}{4 f^2 b^2 + 4bfgq +2g^2q^2}&0&0&0\\
        0&\frac{8b^3f^3q^2\left(bf+gq\right)}{3\left(bf-3gq\right)\left(2b^2f^2-2bfgq-3g^2q^2\right)} & 0 & 0\\
        0 & 0 & \frac{q\left(bf+gq\right)^2}{2g} & 0\\
        0 & 0 & 0 & \frac{q\left(bf+gq\right)^3}{2g\left(-bf+3gq\right)}
	\end{pmatrix},
\end{align*}
whereas, for $b<0$, we obtain,
\begin{align*}
	\mathcal{Y}^x = \begin{pmatrix}
		-\frac{f b^3\left(bf+gq\right)^2}{2gq\left(2 b f+g q\right)}&0&0&0\\
        0&-\frac{8b^3f^3q\left(bf+gq\right)}{3g\left(-bf+3gq\right)\left(2bf+3gq\right)} & 0 & 0\\
        0 & 0 & \frac{q\left(bf+gq\right)^2}{2g} & 0\\
        0 & 0 & 0 & \frac{q\left(bf+gq\right)^3}{2g\left(-bf+3gq\right)}
	\end{pmatrix},
\end{align*}
Then the scaling matrix $T^x$ is the square root of $\mathcal{Y}^x$. For $b>0$ it is given by,
\begin{align*}
	\begin{pmatrix}
		\Theta_1 & 0 & 0 & 0\\
        0&2\sqrt{\frac{2}{3}}\Theta_2 & 0 & 0\\
        0 & 0 & \sqrt{\frac{q\left(bf+gq\right)^2}{2g}} & 0\\
        0 & 0 & 0 & \sqrt{\frac{q\left(bf+gq\right)^3}{2g\left(-bf+3gq\right)}}
	\end{pmatrix}.
\end{align*}
where $\Theta_1=\sqrt{\frac{f b^3\left(bf+gq\right)^2}{4 f^2 b^2 + 4bfgq +2g^2q^2}}$ and $\Theta_2=\sqrt{\frac{b^3f^3q^2\left(bf+gq\right)}{\left(bf-3gq\right)\left(2b^2f^2-2bfgq-3g^2q^2\right)}}$. For $b<0$, the scaling matrix $T^x$ is given by
\begin{align*}
	\begin{pmatrix}
		\frac{\Theta_3}{\sqrt{2}} & 0 & 0 & 0\\
        0&2\sqrt{\frac{2}{3}}\Theta_4 & 0 & 0\\
        0 & 0 & \sqrt{\frac{q\left(bf+gq\right)^2}{2g}} & 0\\
        0 & 0 & 0 & \sqrt{\frac{q\left(bf+gq\right)^3}{2g\left(-bf+3gq\right)}}
	\end{pmatrix}.
\end{align*}
where $\Theta_3=\sqrt{-\frac{f b^3\left(bf+gq\right)^2}{gq\left(2 b f+g q\right)}}$ and $\Theta_4=\sqrt{-\frac{b^3f^3q\left(bf+gq\right)}{g\left(-bf+3gq\right)\left(2bf+3gq\right)}}$.
%--------------------------------------------------------------------
%--------------------------------------------------------------------
%\subsubsection*{Entropy-scaled right eigenvectors in $x$ and $y$ direction}\label{ESRE}
Finally, we present the complete expressions for the entropy-scaled right eigenvector matrices in the $x$ and $y$-directions for $b>0$, which are given below:
\begin{align*}
	\begin{pmatrix}
		-\frac{f\Theta_1}{b} & \frac{\Theta_2}{\sqrt{6}}\left(\frac{fb-3gq}{bq}\right) & 0 & 0\\
        \Theta_1 & \frac{\Theta_2}{\sqrt{6}}\left(\frac{fb-3gq}{fq}\right) & 0 & 0\\
        0 & \sqrt{\frac{2}{3}}\frac{2g\Theta_2}{q} & -\sqrt{\frac{g\left(bf+gq\right)^2}{2q}} & \sqrt{\frac{g\left(bf+gq\right)^3}{2q\left(-bf+3gq\right)}}\\
        0 & \sqrt{\frac{2}{3}}~2\Theta_2 & \sqrt{\frac{q\left(bf+gq\right)^2}{2g}} & \sqrt{\frac{q\left(bf+gq\right)^3}{2g\left(-bf+3gq\right)}}
	\end{pmatrix}.
\end{align*}
and the entropy-scaled right eigenvector matrices in the $x$ and $y$-directions for $b<0$ is are given below:
\begin{align*}
	\begin{pmatrix}
		-\frac{f}{b\sqrt{2}}\Theta_3 & \frac{\Theta_4}{\sqrt{6}}\left(\frac{fb-3gq}{bq}\right) & 0 & 0\\
        \frac{\Theta_3}{\sqrt{2}} & \frac{\Theta_4}{\sqrt{6}}\left(\frac{fb-3gq}{fq}\right) & 0 & 0\\
        0 & \sqrt{\frac{2}{3}}\frac{2g\Theta_4}{q} & -\sqrt{\frac{g\left(bf+gq\right)^2}{2q}} & \sqrt{\frac{g\left(bf+gq\right)^3}{2q\left(-bf+3gq\right)}}\\
        0 & \sqrt{\frac{2}{3}}~2\Theta_4 & \sqrt{\frac{q\left(bf+gq\right)^2}{2g}} & \sqrt{\frac{q\left(bf+gq\right)^3}{2g\left(-bf+3gq\right)}}
	\end{pmatrix}.
\end{align*}

\end{document}